\documentclass[onefignum,onetabnum]{siamart190516}

\usepackage{lipsum}
\usepackage{amsfonts}
\usepackage{graphicx}
\usepackage{epstopdf}
\usepackage{algorithmic}
\ifpdf
  \DeclareGraphicsExtensions{.eps,.pdf,.png,.jpg}
\else
  \DeclareGraphicsExtensions{.eps}
\fi

\usepackage{amsmath}
\usepackage{amssymb}
\usepackage{theorem}
\usepackage{mathrsfs} 
\usepackage{pstricks-add}
\usepackage{pst-plot}
\usepackage{pst-math}
\usepackage{graphicx}
\usepackage{tabularx}
\usepackage{enumitem}
\usepackage{empheq}
\usepackage{srcltx}
\usepackage{pgfplots}
\pgfplotsset{compat=1.15}
\usepackage{subcaption}

\usepackage{array}
\usepackage{multirow}

\renewcommand{\leq}{\ensuremath{\leqslant}}
\renewcommand{\geq}{\ensuremath{\geqslant}}
\renewcommand{\le}{\ensuremath{\leqslant}}
\renewcommand{\ge}{\ensuremath{\geqslant}}
\newcommand{\minimize}[2]{\ensuremath{\underset{\substack{{#1}}}%
{\text{\rm minimize}}\;\;#2}}

\newcommand{\scal}[2]{{\left\langle{{#1}\mid{#2}}\right\rangle}}

\newcommand{\HH}{\ensuremath{{\mathcal H}}}

\newcommand{\GG}{\ensuremath{{\mathcal G}}}

\newcommand{\Id}{\ensuremath{\operatorname{Id}}}

\newcommand{\RR}{\ensuremath{\mathbb{R}}}

\newcommand{\RPP}{\ensuremath{\left]0,+\infty\right[}}

\newcommand{\RX}{\ensuremath{\left]-\infty,+\infty\right]}}

\newcommand{\NN}{\ensuremath{\mathbb N}}

\newcommand{\ran}{\ensuremath{\text{\rm ran}\,}}

\usepackage{pifont}
\newcommand{\net}{\widetilde{J}}
\newcommand{\netav}{Q}
\newcommand{\jac}{\pmb{\nabla}}
\usepackage{siunitx}
\usepackage{multirow}
\usepackage{array}   
\newcolumntype{L}{>{$}l<{$}} 
\newcolumntype{C}{>{$}c<{$}} 

\renewcommand\theenumi{(\roman{enumi})}
\renewcommand\theenumii{(\alph{enumii})}

\newsiamthm{remark}{Remark}
\newsiamthm{notation}{Notation}
\newsiamthm{assumption}{Assumption}
\newsiamthm{claim}{Claim}
\newsiamthm{example}{Example}
\newsiamthm{model}{Model}

\newsiamremark{hypothesis}{Hypothesis}
\crefname{hypothesis}{Hypothesis}{Hypotheses}

\headers{Learning Maximally Monotone Operators for Image Recovery}{J.-C. Pesquet, A. Repetti, M. Terris and Y. Wiaux}

\title{Learning Maximally Monotone Operators\\ for Image Recovery\thanks{Submitted to the editors 12/23/2020.
\funding{
M. Terris would like to thank Heriot-Watt University for the PhD funding. This work was supported by the UK Engineering and Physical Sciences Research Council (EPSRC) under grant EP/T028270/1.
The work of J.-C. Pesquet was supported by Institut Universitaire de France and the ANR Chair in AI BRIGEABLE.}}}

\author{
Jean-Christophe Pesquet\thanks{
Universit\'e Paris-Saclay, Inria, 
Center for Visual Computing, Gif sur Yvette, France (\email{jean-christophe.pesquet@centralesupelec.fr}).}
\and Audrey Repetti\thanks{Second corresponding author. School of Mathematics and Computer Sciences and School of Engineering and Physical Sciences, Heriot-Watt University, Edinburgh, UK. Maxwell Institute for Mathematical
Sciences, Bayes Centre, Edinburgh, UK (\email{a.repetti@hw.ac.uk}).}
\and Matthieu Terris\thanks{First corresponding author. School of Engineering and Physical Sciences, Heriot-Watt University, Edinburgh, UK (\email{mt114@hw.ac.uk}).}
\and Yves Wiaux\thanks{School of Engineering and Physical Sciences, Heriot-Watt University, Edinburgh, UK (\email{y.wiaux@hw.ac.uk}).}}

\usepackage{amsopn}

\ifpdf
\hypersetup{
  pdftitle={Learning Maximally Monotone Operators for Image Recovery},
  pdfauthor={J.-C. Pesquet, A. Repetti, M. Terris, and Y. Wiaux}
}
\fi

\begin{document}

\maketitle

\begin{abstract}
We introduce a new paradigm for solving regularized variational problems. These are typically formulated to address ill-posed inverse problems encountered in signal and image processing. The objective function is traditionally defined by adding a regularization function to a data fit term, which is subsequently minimized by using iterative optimization algorithms. Recently, several works have proposed to replace the operator related to the regularization by a more sophisticated denoiser. These approaches, known as plug-and-play (PnP) methods, have shown excellent performance. Although it has been noticed that, under some Lipschitz properties on the denoisers, the convergence of the resulting algorithm is guaranteed, little is known about characterizing the asymptotically delivered solution.  In the current article, we propose to address this limitation.
More specifically, instead of employing a functional regularization, we perform an operator regularization, where a maximally monotone operator (MMO) is learned in a supervised manner. This formulation is flexible as it allows the solution to be characterized through a broad range of variational inequalities, and it includes convex regularizations as special cases. From an algorithmic standpoint, the proposed approach consists in replacing the resolvent of the MMO by a neural network (NN). 
We present a universal approximation theorem proving that nonexpansive NNs are suitable models for the resolvent of a wide class of MMOs. The proposed approach thus provides a sound theoretical framework for analyzing the asymptotic behavior of first-order PnP algorithms. In addition, we propose a numerical strategy to train NNs corresponding to resolvents of MMOs.  We apply our approach to image restoration problems and demonstrate its validity in terms of both convergence and quality.
\end{abstract}



\section{Introduction}
\label{Sec:Intro}

In many problems in data science, in particular when dealing with inverse problems, a variational approach is adopted which amounts to
\begin{equation}\label{e:min}
\minimize{x\in \HH}
{f(x)+g(x)}
\end{equation}
where $\HH$ is the underlying data space, here assumed to be a real Hilbert space, $f\colon\HH \to \RX$ is a data fit (or data fidelity) term related to some available data $z$ (observations), and $g\colon \HH\to \RX$ is some regularization function. 
The data fit term is often derived from statistical considerations on the observation model through the maximum likelihood principle. For many standard noise distributions, the negative log-likelihood corresponds to a smooth function (e.g. Gaussian, Poisson-Gauss, or logistic distributions).
The regularization term is often necessary to avoid overfitting or to overcome ill-posedness problems. A vast literature has been developed on the choice of this term. It often tends to promote the smoothness of the solution or to enforce its sparsity by adopting a functional analysis viewpoint. 
Good examples of such regularization functions are the total variation semi-norm \cite{rudin1992nonlinear} and its various extensions \cite{bredies2010GTV,condat2014SLTV}, and penalizations based on wavelet (or ``x-let'') frame representations \cite{daubechies2004}. Alternatively, a Bayesian approach can be followed where this regularization is viewed as the negative-log of some prior distribution, in which case the minimizer of the objective function in \eqref{e:min} can be understood as a Maximum A Posteriori (MAP) estimator. In any case, the choice of this regularization introduces two main roadblocks. First, the function $g$ has to be chosen so that the minimization problem in \eqref{e:min} be tractable, which limits its choice to relatively simple forms. 
Secondly, the definition of this function involves some parameters which need to be set. The simplest case consists of a single scaling parameter usually called the regularization factor, the choice of which is often very sensitive on the quality of the results. 
Note that, in some works, this regularization function is the indicator function of some set encoding some smoothness or sparsity constraint. For example, it can model an upper bound on some functional of the discrete gradient of the sought signal, this bound playing then a role equivalent to a regularization parameter \cite{combettes2004image}. 
Using an indicator function can also model standard constraints in some image restoration problems, where the image values are bounded \cite{abdulaziz2019wideband,birdi2018sparse}.
In order to solve such problems, a wide range of proximal splitting algorithms have been developed in the past decades \cite{combettes2011proximal}.

By denoting by $\Gamma_{0}(\HH)$ the class of lower-semicontinuous convex functions from $\HH$ to $\RX$ with
a nonempty domain, let us now assume that 
both $f$ and $g$ belong to $\Gamma_{0}(\HH)$. The Moreau
subdifferentials of these functions will be denoted by $\partial f$
and $\partial g$, respectively. Under these convexity assumptions, if 
\begin{equation}\label{e:zer}
0 \in \partial f(x) + \partial g(x),
\end{equation}
then $x$ is a solution to the minimization problem \eqref{e:min}. Actually,
under mild qualification conditions the sets of solutions to \eqref{e:min}
and \eqref{e:zer} coincide \cite{bauschke2017convex}. By reformulating the original optimization problem
under the latter form, we have moved to the field of variational inequalities. 
Interestingly, it is a well-established fact that the subdifferential of a function
in $\Gamma_{0}(\HH)$ is a maximally monotone operator (MMO), which means
that \eqref{e:zer} is a special case of the following monotone inclusion
problem:
\begin{equation}\label{e:mon}
\mbox{Find $x\in \HH$ such that }
0 \in \partial f(x) + A(x),
\end{equation}
where $A$ is an MMO. We recall that a multivalued operator $A$ defined on $\HH$ is maximally monotone if and only if, for every $(x_{1},u_{1})\in \HH^2$,
\begin{equation}\label{e:MMOdef}
u_{1}\in Ax_{1} \quad \Leftrightarrow \quad (\forall x_{2}\in \HH)(\forall u_{2}\in A x_{2}) \scal{x_{1}-x_{2}}{u_{1}-u_{2}} \ge 0,
\end{equation}
see e.g. \cite{bauschke2017convex, combettes2020fixed} for a background on monotone operator theory. 
Actually the class of monotone inclusion problems is much wider than the class of convex optimization problems and, in particular, includes saddle point problems and game theory equilibria \cite{combettes2018monotone}. 
It is also worth noticing that many existing algorithms for solving convex optimization problems are particular cases of algorithms designed for solving monotone inclusion problems. The latter often leverage the resolvent operator of $A$ in \eqref{e:mon}, which reduces to a proximal operator when $A=\partial g$ as in \eqref{e:zer} \cite{bauschke2017convex}.
This suggests that it is more flexible, and probably more efficient, to substitute \eqref{e:mon} for \eqref{e:min} in problems encountered in data science.
In other words, instead of performing a functional regularization, we can introduce an operator regularization through
the maximally monotone mapping $A$. Although this extension of  \eqref{e:min} may appear both natural and elegant, it induces
a high degree of freedom in the choice of the regularization strategy. However, if we except the standard case when $A=\partial g$, it is 
hard to have a good intuition about how to make a relevant choice for $A$. To circumvent this difficulty, our proposed approach will
consist in learning $A$ in a supervised manner by using some available dataset in the targeted application. 

Since a MMO is fully characterized by its resolvent, we propose in this paper to learn the resolvent of $A$. We show that this can be formulated as learning a denoiser with an appropriate 1-Lipschitz condition. Thus, our approach enters into the family of so-called plug-and-play (PnP) methods, where one replaces the proximity operator in an optimization algorithm with a denoiser, e.g. a denoising neural network (NN) \cite{zhang2017learning}.
It is worth mentioning that by doing so, any algorithm whose proof is based on MMO theory, e.g., Forward-Backward (FB), Douglas-Rachford, Peaceman-Rachford, primal-dual approaches, and more \cite{bauschke2017convex, combettes2011proximal, condat2013primal, komodakis2015playing, vu2013splitting} can be turned into a PnP algorithm.
To ensure the convergence of such PnP algorithms, according to fixed point theory (under mild conditions) it is sufficient for the denoiser to be firmly nonexpansive \cite{combettes2020fixed}. 
In this context, \eqref{e:mon} offers an elegant characterization of the solution.

Unfortunately, most pre-defined denoisers do not satisfy this assumption, and learning a firmly nonexpansive denoiser remains challenging \cite{ryu2019plug,terris2020building}. The main bottleneck is the ability to tightly constrain the Lipschitz constant of a NN. During the last years, several works proposed to control the Lipschitz constant (see e.g. \cite{anil2018sorting, bibi2018deep, cisse2017parseval, hertrich2020convolutional, miyato2018spectral, ryu2019plug, sedghi2018singular, terris2020building, yoshida2017spectral}). Nevertheless, only few of them are accurate enough to ensure the convergence of the associated PnP algorithm and often come at the price of strong computational and architectural restrictions (e.g., absence of residual skip connections) \cite{bibi2018deep, hertrich2020convolutional, ryu2019plug, terris2020building}. 
The method proposed in \cite{bibi2018deep} allows a tight control of convolutional layers, but in order to ensure the nonexpansiveness of the resulting architecture, one cannot use residual skip connections, despite their wide use in NNs for denoising applications. 
In \cite{hertrich2020convolutional}, the authors propose to train an averaged NN by projecting the full convolutional layers on the Stiefel manifold and showcase the usage of their network in a PnP algorithm. Yet, the architecture proposed by the authors remains constrained by proximal calculus rules.
In our previous work \cite{terris2020building}, we proposed a method to build firmly nonexpansive convolutional NNs; to the best of our knowledge, this was the first method ensuring the firm nonexpansiveness of a denoising NN. However, the resulting architecture was strongly constrained and did not improve over the state-of-the-art. Since building firmly nonexpansive denoisers is difficult, many works on PnP methods leverage ADMM algorithm  which may appear easier to handle in practice \cite{ryu2019plug}. However, the convergence of ADMM requires restrictive conditions on the involved linear operators~\cite{komodakis2015playing}.

Alternative approaches have been proposed in the literature to circumvent these difficulties. Ryu et al.~\cite{ryu2019plug} proposed the first convergent NN-based PnP algorithm in a more general framework, namely RealSN. Nevertheless, in this work, the authors rely on \cite[Assumption A]{ryu2019plug}, which may potentially lead to expansive networks (see \cite[Lemma 9 \& 10]{ryu2019plug}), and unstable PnP algorithms. The regularization by denoising (RED) approach \cite{ahmad2020plug, cohen2020regularization} investigates situations when the PnP algorithm converges to a solution to a minimization problem.
In particular, the minimum mean square error (MMSE) denoiser can be  employed \cite{ahmad2020plug, xu2020provable}. However, as underlined by the authors, the denoising NN is only an approximation to  the MMSE regressor. The authors of \cite{cohen2020regularization} proposed RED-PRO, an extension of RED  within the framework of fixed point theory. Under a demi-contractivity assumption, the authors show that the PnP algorithm converges to a solution to a constrained minimization problem. 

The contribution of this paper is twofold. We first provide a universal approximation theorem for a wide class of monotone operators. In particular, we show that one can approximate the resolvent of a \emph{stationary} MMO as closely as desired with a firmly nonexpansive NN. We then propose a practical framework to impose the firm nonexpansiveness of our NN.
To do so, we regularize the training loss of the NN with the spectral norm of the Jacobian of a suitable nonlinear mapping. Although the resulting NN could be plugged into a variety of iterative algorithms, our work is focused on the standard FB algorithm. We illustrate the convergence of the corresponding PnP scheme in an image restoration problem. We show that our method compares positively in terms of quality to both state-of-the-art PnP methods and regularized optimization approaches.

This article is organized as follows. In \cref{sec:2}, we recall how MMOs can be mathematically characterized and explain how their resolvent can be modeled by an averaged residual neural network. We also establish that NNs are generic models for a wide class of MMOs.
In \cref{sec:3}, we show the usefulness of learning MMOs in the context of plug-and-play (PnP) first-order algorithms employed for solving inverse problems. We also describe the training approach which has been adopted. In \cref{sec:4}, we provide illustrative results for the restoration of monochromatic and color images.
Finally, some concluding remarks are made in \cref{sec:5}.

\noindent \textbf{Notation}: Throughout the article, we will denote by $\|\cdot\|$ the norm endowing any real Hilbert space $\HH$. The same notation (being clear from the context) will be used to denote the norm of a bounded linear operator $L$ from  $\HH$ to some real Hilbert space $\GG$, that is $\| L\| = \sup_{x\in \HH\setminus\{0\}} \|Lx\|/\|x\|$. The inner product of $\HH$ associated to $\|\cdot\|$ will be denoted by $\scal{\cdot}{\cdot}$, here again without making explicit the associated space. 
Let $D$ be a subset of $\HH$ and $T\colon D \to \HH$. The operator $T$ is \emph{$\mu$-Lipschitzian} for $\mu>0$ if, for every $(x,y)\in D^2$, $\|Tx-Ty\|\leq \mu \|x-y\|$.
If $T$ is $1$-Lipschitzian, its is said to be \emph{nonexpansive}. The operator $T$ is \emph{firmly nonexpansive} if, for every $(x,y)\in D^2$, $\|Tx-Ty\|^2 \leq \scal{x-y}{Tx-Ty}$.
We denote by $2^\HH$ be the power set of $\HH$, that is the set of all subsets of $\HH$. Let $A\colon \HH \to 2^{\HH}$ be a multivariate operator, i.e., for every $x\in \HH$, $A(x)$ is a subset of $\HH$. The \emph{graph} of $A$ is defined as $\operatorname{gra}A=\{(x,u)\in \HH^2 \mid u\in Ax\}$.
The operator $A\colon \HH \to 2^{\HH}$ is \emph{monotone} if, for every $(x,u)\in \operatorname{gra}A$ and $(y,v)\in \operatorname{gra}A$, $\scal{x-y}{u-v} \geq 0$, and \emph{maximally-monotone} if \eqref{e:MMOdef} holds, for every $(x_{1},u_{1})\in \HH^2$. The resolvent of $A$ is $J_{A} = (\Id+ A)^{-1}$, where the inverse is here defined in the sense of the inversion of the graph of the operator. For further details on monotone operator theory, we refer the reader to \cite{bauschke2017convex, combettes2020fixed}.

\section{Neural network models for maximally monotone operators}
\label{sec:2}

\subsection{A property of maximally monotone operators}

We first introduce a main property that will link explicitly MMOs to nonexpansive operators, noticing that any multivalued operator on $\HH$ is fully characterized by its resolvent.

\begin{proposition}\label{prop:basmaxmon}
Let $A\colon \HH \to 2^{\HH}$. $A$ is a MMO if and only if
there exists a nonexpansive (i.e. 1-Lipschitzian) operator $Q\colon \HH \to \HH$ such
that
\begin{align}\label{e:JAQ}
J_{A}\colon \HH &\to \HH\colon
x \mapsto \frac{x+Q(x)}{2},
\end{align}
that is
\begin{equation}\label{e:AQ}
A = 2(\Id+Q)^{-1}-\Id.
\end{equation}
\end{proposition}

\begin{proof}
This result is a direct consequence of Minty's theorem and the fact that any firmly nonexpansive operator 
can be expressed as the arithmetic mean of the identity operator and some nonexpansive operator $Q$ (see \cite[Proposition 23.8]{bauschke2017convex}).
\eqref{e:AQ}~is deduced by inverting~\eqref{e:JAQ}.
\end{proof}

The above result means that the class of MMOs can be derived from the class of nonexpansive mappings. The focus should therefore turn on how to model operators in the latter class with neural networks. 

\subsection{Nonexpansive neural networks}

Our objective will next be to derive a parametric model for the nonexpansive operator $Q$ in \eqref{e:JAQ}. Due to their oustanding approximation capabilities, neural networks appear as good choices for building such models. We will restrict our attention to feedforward NNs. 

\begin{model}   \label{mod:feedfwd}
Let $(\HH_{m})_{0 \le m \le M}$ be real Hilbert spaces
such that $\HH_{0}= \HH_{M}=\HH$. A feedforward NN having $M$ layer and both input and ouput in $\HH$ can be
seen as a composition of operators:
\begin{equation}\label{e:defNN}
 Q=T_M \cdots T_1,
\end{equation}
where 
\begin{align}\label{e:defNNlayeri}
(\forall m\in\{1,\ldots,M\})\quad 
T_m\colon\HH_{m-1}&\to\HH_m\colon
 x\mapsto R_m(W_mx+b_m).
\end{align}
At each layer $m  \in\{1,\ldots,M\}$, 
$R_{m}\colon \HH_{m} \to \HH_{m}$ is a nonlinear activation operator, \linebreak $W_{m}\colon \HH_{m-1}\to \HH_{m}$ is a bounded 
linear operator corresponding to the weights of the network, and $b_{m}\in \HH_{m}$ is a bias parameter vector.
\end{model}

In the remainder, we will use the following notation:
\begin{notation}
Let $V$ and $V'$ be nonempty subsets of some Euclidean space and let $\mathcal{N}_\mathcal{F}(V,V')$ denote the class of nonexpansive feedforward NNs with inputs in $V$ and outputs in $V'$ built from a given dictionary $\mathcal{F}$ of allowable activation operators. 
\end{notation}
Also, we will make the following assumption:
\begin{assumption}\label{as:actifun}
The identity operator as well as the sorting operator performed on blocks of size 2 belong to dictionary $\mathcal{F}$. 
\end{assumption}
In other words, a network in $\mathcal{N}_\mathcal{F}(V,V')$ can be linear, or it can be built by using max-pooling with blocksize 2 and any other kind of activation function, say some given function $\rho\colon \RR \to \RR$,  operating componentwise in some of its layers, provided that the resulting structure is 1-Lipschitzian.

The main difficulty is to design such a feedforward NN so that $Q$ in \eqref{e:defNN} has a Lipschitz constant smaller or equal to 1. 
An extensive literature has been devoted to the estimation of Lipschitz constants of NNs \cite{anil2018sorting,Scam18,Szeg13}, but the main goal was different from ours since these works were motivated by robustness issues in the presence of adversarial perturbations \cite{fazlyab2019efficient,latorre2020lipschitz,neacsu2020accuracy,Szeg13}.

In this work, we propose to enforce the nonexpansiveness of $Q$ by adding a regularization during the training of the NN as explained in \cref{sec:3}. Other conditions ensuring the firm nonexpansiveness of Model~\ref{mod:feedfwd} can be found in Appendix~\ref{apx:nonexpansive}.

\subsection{Stationary maximally monotone operators}

The one-to-one correspondence between a MMO and a nonexpansive operator from \cref{prop:basmaxmon} raises the following question: can one approximate the resolvent of a MMO as closely as desired by a firmly nonexpansive neural network structure? The remainder of this work aims to answer this question in the case of a subclass of MMOs, namely those which are stationary, as defined hereunder.

\begin{definition}\label{d:maxmonstat}
Let $(\HH_k)_{1\le k \le K}$ be real Hilbert spaces.
An operator $A$ defined on the product space space $\HH = \HH_{1}\times \cdots \times \HH_{K}$ will be said to be a \emph{stationary} MMO
if its resolvent $J_{A}$ is an operator from $\HH$ to $\HH$ such that, for every $k\in \{1,\ldots,K\}$, there exists a bounded linear operator $\Pi_{k}\colon \HH \to \HH_{k}$ and a self-adjoint nonnegative operator $\Omega_{k}\colon \HH \to \HH$
such that
\begin{equation}\label{e:stat1} 
\big(\forall (x,y) \in \HH^2\big)
\quad \|\Pi_{k}\big(2J_{A} (x)-x -2J_{A} (y)+y\big)\|^2 \le \scal{x-y}{\Omega_{k}(x-y)}
\end{equation}
with
\begin{align} 
&\sum_{k=1}^K \Pi_{k}^{*}\,\Pi_{k} = \Id  \label{e:stat21}\\
&\Big\|\sum_{k=1}^K \Omega_{k}\Big\| \le 1. \label{e:stat22}
\end{align}
\end{definition}

Immediate consequences of this definition are given below.
In particular, we will see that stationary MMOs define a subclass of the set of MMOs. 
\begin{proposition}\label{prop:immSMMO}
Let $(\HH_k)_{1\le k \le K}$ be real Hilbert spaces and
let $\HH = \HH_{1}\times \cdots \times \HH_{K}$.
Let $A\colon \HH\to 2^{\HH}$.
\begin{enumerate}
    \item \label{prop:immSMMOi}  If $A$ is a stationary MMO on $\HH$, then it is maximally monotone.
    \item \label{prop:immSMMOii} Assume that \eqref{e:stat21} is satisfied  where,
    for every $k\in \{1,\ldots,K\}$, $\Pi_{k}\colon \HH \to \HH_{k}$  is a bounded linear operator.
    If $\ran(A+\Id)=\HH$ and
    \begin{equation} 
    (\forall (p,q)\in \HH^2)(\forall p'\in A(p))(\forall q'\in A(q))\;\;\scal{\Pi_k (p-q)}{\Pi_k (p'-q')} \ge 0,
    \label{e:monAk}
    \end{equation}
    then $A$ is a stationary MMO.
\end{enumerate}
\end{proposition}
\begin{proof}

\noindent\ref{prop:immSMMOi}: 
Let $A$ be a stationary MMO defined on $\HH$. Summing over $k$ in
\eqref{e:stat1} yields, for every $(x,y)\in \HH^2$,
\begin{multline}
    \scal{2J_{A} (x)-x -2J_{A} (y)+y}{\Big(\sum_{k=1}^K\Pi_k^*\,\Pi_k\Big)
    (2J_{A} (x)-x -2J_{A} (y)+y)}\\
    \le \scal{x-y}{\sum_{k=1}^K\Omega_{k}(x-y)}.
\end{multline}
It thus follows from
\eqref{e:stat21}, \eqref{e:stat22}, and the nonnegativity of $(\Omega_k)_{1\le k\le K}$ that
\begin{equation}
    \|2J_{A} (x)-x -2J_{A} (y)+y\|^2
    \le \Big\|\sum_{k=1}^K\Omega_{k}\Big\| \|x-y\|^2 \le \|x-y\|^2.
\end{equation}
This shows that $2 J_A-\Id$ is a nonexpansive operator. Hence, based on Proposition~\ref{p:propnonexpan}, $A$ is an MMO.

\noindent\ref{prop:immSMMOii}: 
Let $k$ be an arbitrary integer in $\{1,\ldots,K\}$. \eqref{e:monAk} can be reexpressed as
    \begin{multline} \label{e:monAkbis}
    (\forall (p,q)\in \HH^2)(\forall p'\in A(p))(\forall q'\in A(q))\\
    \scal{\Pi_k^*\Pi_k (p-q)}{p'-q'+p-q} \ge \scal{\Pi_k^*\Pi_k (p-q)}{p-q}.
    \end{multline}
    In particular, this inequality holds if $p\in J_A(x)$ and $q\in J_A(y)$ where $x$ and $y$ are arbitrary elements of $\HH$. Then, by definition of $J_A$, we have $x-p\in A(p)$, $y-q\in A(q)$, and \eqref{e:monAkbis} yields
    \begin{equation}\label{e:firmJAk}
    \scal{\Pi_k^*\Pi_k (p-q)}{x-y} \ge \scal{\Pi_k^*\Pi_k (p-q)}{p-q}.
    \end{equation}
    By summing over $k$ and using \eqref{e:stat21}, it follows that 
    $J_A$ is firmly nonexpansive and it is thus single valued.
    \eqref{e:firmJAk} is then equivalent to
    \begin{equation}
    \|\Pi_{k}\big(2J_{A} (x)-x -2J_{A} (y)+y\big)\|^2 \le \scal{x-y}{\Pi_k^* \Pi_k(x-y)}.
    \end{equation}
    This shows that
    Inequality \eqref{e:stat1} holds with $\Omega_k=\Pi_k^*\Pi_k$.
    Since \eqref{e:stat22} is then obviously satisfied, $A$ is a stationary MMO.
\end{proof}

In order to demonstrate the versatility of stationary MMOs, 
we feature a few examples of such operators. The validity of these examples is proved in Appendix~\ref{apx:statMMO}.

\begin{example}\label{ex:stat:separabilityter}
For every $k\in \{1,\ldots,K\}$, let
$B_{k}$ be an MMO defined on a real Hilbert space $\HH_{k}$ and let $B$ be the operator defined as
\begin{equation}\label{e:sepAter}
(\forall x = (x^{(k)})_{1\le k \le K} \in \HH = \HH_{1}\times \cdots \times \HH_{K}) \quad 
B(x) = B_{1}(x^{(1)}) \times \cdots \times B_{K}(x^{(K)}).
\end{equation}
Let $U \colon \HH \to \HH$ be a unitary linear operator. Then
$A = U^*B U$
is a stationary MMO.
\end{example}

\begin{example}\label{ex:stat:separabilityf}
For every $k\in \{1,\ldots,K\}$, let $\varphi_k\in \Gamma_0(\RR)$, 
and let the function $g$ be defined as 
\begin{equation}\label{e:sepAf}
(\forall x = (x^{(k)})_{1\le k \le K} \in \RR^K)
\quad 
g(x) = \sum_{k=1}^K \varphi_k(x^{(k)}).
\end{equation}
Let $U \in \RR^{K\times K}$ be an orthogonal matrix.
Then the subdifferential of $g\circ U$ is  a stationary MMO.
\end{example}

\begin{example}\label{ex:stat:linearity}
Let $(\HH_{k})_{1\le k \le K}$ be real Hilbert spaces and let
$B$ be a bounded linear operator from $\HH = \HH_{1}\times \cdots \times \HH_{K}$
to $\HH$ such that one of the following conditions holds:
\begin{enumerate}
\item $B+B^*$ is nonnegative
\item $B$ is skewed
\item $B$ is cocoercive.
\end{enumerate}
Let $c\in \HH$.
Then the affine operator $A\colon \HH\to \HH\colon x \mapsto Bx+c$ is a stationary MMO.
\end{example}

\begin{example}\label{ex:stat:inverse}
Let $(\HH_k)_{1\le k \le K}$ be real Hilbert spaces, let
$\HH = \HH_{1}\times \cdots \times \HH_{K}$, and let
$A\colon \HH \to 2^{\HH}$ be a stationary MMO. Then its inverse $A^{-1}$ is a stationary MMO.
\end{example}

\begin{example}\label{ex:stat:lin}
Let $(\HH_k)_{1\le k \le K}$ be real Hilbert spaces, let
$\HH = \HH_{1}\times \cdots \times \HH_{K}$, and let
$A\colon \HH \to 2^{\HH}$ be a stationary MMO. Then, for every $\rho \in \RR\setminus\{0\}$, $\rho A(\cdot/\rho)$ is a stationary MMO.
\end{example}

\subsection{Universal approximation theorem}

In this section we provide one of the main contributions of this article, consisting in a universal approximation theorem for stationary MMOs defined on $\HH = \RR^K$. 
To this aim, we first need to introduce useful results, starting by recalling the definition of a lattice.
\begin{definition}
A set $\mathcal{L}_E$ of functions from a set $E$ to $\RR$ is said to be a \emph{lattice} if, for every $(h^{(1)},h^{(2)})\in \mathcal{L}_E^2$,
$\min\{h^{(1)},h^{(2)}\}$ and $\max\{h^{(1)},h^{(2)}\}$ belong to $\mathcal{L}_E$. A sub-lattice of $\mathcal{L}_E$ is a lattice included in $\mathcal{L}_E$.
\end{definition}

This notion of lattice is essential in the variant of the Stone-Weierstrass theorem provided below.

\begin{proposition}{\rm \cite{anil2018sorting}} \label{prop:SW}
Let $(E,d)$ be a compact metric space with at least two distinct points. Let $\mathcal{L}_E$ be a sub-lattice of $\mathrm{Lip}_{1}(E,\RR)$, the class of 1-Lipschtzian (i.e. nonexpansive) functions from $E$ to $\RR$. 
Assume that, for every $(u,v)\in E^2$ with $u\neq v$ and, for every $(\zeta,\eta) \in \RR^{2}$
such that $|\zeta-\eta| \le d(u,v)$, there exists a function $h\in \mathcal{L}_E$ such that $h(u) = \zeta$
and $h(v) = \eta$. Then $\mathcal{L}_E$  is dense in  $\mathrm{Lip}_{1}(E,\RR)$ for the uniform norm.
\end{proposition}

This allows us to derive the following approximation result that will be instrumental to prove our main result. 
\begin{corollary}\label{co:SW}
Let $V$ be a subspace of $\RR^K$ and let $h\in \mathrm{Lip}_{1}(V,\RR)$. Let $E$ be a compact subset
of $V$. Then, for every $\varepsilon \in \RPP$, there exists $h_{\varepsilon}\in \mathcal{N}_{\mathcal{F}}(V,\RR)$, where $\mathcal{F}$
is any dictionary of activation function satisfying \cref{as:actifun},
such that
\begin{equation}\label{e:approxSW}
(\forall x \in E)\quad |h(x)-h_{\varepsilon}(x)| \le \varepsilon.
\end{equation}
\end{corollary} 
\begin{proof}
First note that $\mathcal{N}_{\mathcal{F}}(V,\RR)$ is a lattice. Indeed, if $h^{(1)}\colon V\to \RR$ and $h^{(2)}\colon V\to \RR$ are 1-Lipschitzian, then $\min\{h^{(1)},h^{(2)}\}$ and $\max\{h^{(1)},h^{(2)}\}$ are 1-Lipschitzian.
In addition, if $h^{(1)}$ and $h^{(2)}$ are elements in $\mathcal{N}_{\mathcal{F}}(V,\RR)$, then by applying sorting operations on the two outputs of these two networks, $\min\{h^{(1)},h^{(2)}\}$ and $\max\{h^{(1)},h^{(2)}\}$ are generated.
Each of these outputs can be further selected by applying weight matrices either equal to $[1\;\;0]$ or $[0\;\;1]$ as a last operation, so leading to a NN in $\mathcal{N}_{\mathcal{F}}(V,\RR)$. 

Let $E$ be a compact subset of $V$. Assume that $E$ has at least two distinct points. Since $\mathcal{N}_{\mathcal{F}}(V,\RR)$ is a lattice, the set of restrictions to $E$ of elements in $\mathcal{N}_{\mathcal{F}}(V,\RR)$ is a sub-lattice $\mathcal{L}_E$ of $\mathrm{Lip}_{1}(E,\RR)$. 
In addition, let $(u,v)\in E^{2}$ with $u\neq v$ and let $(\zeta,\eta) \in \RR^{2}$ be such that $|\zeta-\eta| \le \|u-v\|$. Set $h\colon V \to \RR \colon x \mapsto w^\top (x-v) + \eta$ where $w=(\zeta-\eta)(u-v)/\|u-v\|^{2}$. 
Since $\|w\| = |\zeta-\eta|/\|u-v\| \le 1$, $h$ is a linear network in $\mathcal{N}_{\mathcal{F}}(V,\RR)$ and we have $h(u) = \zeta$ and $h(v)=\eta$. This shows that the restriction of $h$ to $E$ is an element of $\mathcal{L}_E$ satisfying the assumptions of \cref{prop:SW}. It can thus be deduced from this proposition
that \eqref{e:approxSW} holds.

The inequality also trivially holds if $E$ reduces to a single point $x$ since it is always possible to find a linear network
in $\mathcal{N}_{\mathcal{F}}(V,\RR)$ whose output equals $h(x)$.
\end{proof}

\begin{remark}
This result is valid whatever the norm used on $V$.
\end{remark}

We are now able to state a universal approximation theorem for MMOs defined on $\HH = \RR^K$ (i.e., for every $k\in \{1,\ldots,K\}$, $\HH_{k}=\RR$ in \cref{d:maxmonstat}).
\begin{theorem}\label{t:approxmon}
Let $\HH=\RR^K$. Let $A\colon \HH \to 2^{\HH}$ be a stationary MMO.
For every compact set $S \subset \HH$ and every $\epsilon \in \RPP$, there exists a NN $Q_{\epsilon}\in \mathcal{N}_{\mathcal{F}}(\HH,\HH)$, where $\mathcal{F}$ is any dictionary of activation function satisfying \cref{as:actifun}, such that $A_{\epsilon} = 2(\Id+Q_{\epsilon}\big)^{-1}-\Id$ satisfies the following properties.
\begin{enumerate}
\item\label{t:approxmoni} For every $x \in S$, $\|J_{A}(x)-J_{A_{\epsilon}}(x)\| \le \epsilon$.
\item\label{t:approxmonii} Let $x \in \HH$ and let $y \in A(x)$ be such that $x+y\in S$. Then, there exists $x_{\epsilon}\in \HH$ and $y_{\epsilon} \in A_{\epsilon}(x_{\epsilon})$ such that $\|x-x_{\epsilon}\| \le \epsilon$ and $\|y-y_{\epsilon}\| \le \epsilon$.
\end{enumerate}
\end{theorem}

\begin{proof}
\noindent\ref{t:approxmoni}: If $A\colon \RR^K \to 2^{\RR^{K}}$ is a stationary MMO then it follows from Propositions~\ref{prop:basmaxmon} and \ref{prop:immSMMO}\ref{prop:immSMMOi} that there exists a nonexpansive operator $Q\colon \RR^K \to \RR^{K}$ such that $J_{A}= (\Id+Q)/2$. In addition, according to \cref{d:maxmonstat}, there exist vectors $(p_{k})_{1\le k \le K}$ in $\RR^K$ such that, for every $k \in \{1,\ldots,K\}$,
\begin{equation}\label{e:pQLip}
\big(\forall (x,y) \in \HH^2\big)
\quad |\scal{p_{k}}{Q(x)-Q(y)}|^2 \le \scal{x-y}{\Omega_{k}(x-y)}
\end{equation}
where 
\begin{equation}\label{e:condpk}
\sum_{k=1}^{K} p_{k} p_{k}^\top = \Id
\end{equation}
and $(\Omega_{k})_{1\le k \le K}$ are positive semidefinite matrices in $\RR^{K\times K}$ satisfying \eqref{e:stat22}. 

Set $k\in \{1,\ldots,K\}$ and define $h_{k}\colon x \mapsto \scal{p_{k}}{Q(x)}$. Let $V_{k}$ be the nullspace of $\Omega_{k}$ and let $V_{k}^\perp$ be its orthogonal space. We distinguish the cases when $V_{k}^\perp \neq \{0\}$ and when $V_{k}^\perp = \{0\}$. Assume that $V_{k}^\perp \neq \{0\}$. It follows from \eqref{e:pQLip} that, for every $x\in V_{k}^\perp$ and $(y,z) \in V_{k}^2$, 
\begin{equation}\label{e:hkred}
 h_{k}(x+y) = h_{k}(x+z) = \widetilde{h}_{k}(x)
 \end{equation}
 where $\widetilde{h}_{k}\colon V_{k}^\perp \to \RR$ is such that
 \begin{equation}\label{e:Liphtk}
\big(\forall (x,x') \in (V_{k}^\perp)^2\big)
\quad |\widetilde{h}_{k}(x)-\widetilde{h}_{k}(x')| \le \|x-x'\|_{\Omega_k}
\end{equation}
and $(\forall x \in \RR^{K})$ $\|x\|_{\Omega_k} = \scal{x}{\Omega_{k}x}^{1/2}$. $\|\cdot\|_{\Omega_k}$ defines a norm on $V_{k}^\perp$. Inequality \eqref{e:Liphtk} shows  that $\widetilde{h}_{k}$ is 1-Lipschitzian on $V_{k}^\perp$ equipped with this norm. Let $S$ be a compact subset of $\RR^K$ and let $\operatorname{proj}_{V_{k}^\perp}$ be the orthogonal projection onto $V_{k}^\perp$.
$E_{k} = \operatorname{proj}_{V_{k}^\perp}(S)$ is a compact set and, in view of \cref{co:SW}, for every $\epsilon \in \RR$, there exists $\widetilde{h}_{k,\epsilon}\in \mathcal{N}_{\mathcal{F}}(V_{k}^\perp,\RR)$ such that 
\begin{equation}\label{e:approxSWk}
(\forall x \in E_{k})\quad |\widetilde{h}_{k}(x)-\widetilde{h}_{k,\epsilon}(x)| \le \frac{2\epsilon}{\sqrt{K}}.
\end{equation}
Set now $h_{k,\epsilon}= \widetilde{h}_{k,\epsilon}\circ \operatorname{proj}_{V_{k}^\perp}$. According to \eqref{e:hkred} and \eqref{e:approxSWk}, we have
\begin{align}\label{e:approxhkhkeps}
(\forall x \in S)\quad &|h_{k}(x) - h_{k,\epsilon}(x)|\nonumber\\
&=
|h_{k}(\operatorname{proj}_{V_{k}}(x)+\operatorname{proj}_{V_{k}^\perp}(x)) - h_{k,\epsilon}(\operatorname{proj}_{V_{k}}(x)+\operatorname{proj}_{V_{k}^\perp}(x))|\nonumber\\
&=|\widetilde{h}_{k}(\operatorname{proj}_{V_{k}^\perp}(x)) - \widetilde{h}_{k,\epsilon}(\operatorname{proj}_{V_{k}^\perp}(x)) |\nonumber\\  
&\le \frac{2\epsilon}{\sqrt{K}}.
\end{align}
In addition, by using the Lipschitz property of $\widetilde{h}_{k,\epsilon}$ with respect to norm $\|\cdot\|_{\Omega_k}$, for every $(x,x')\in \RR^K$,
\begin{align}\label{e:Liphkeps}
&\big(h_{k,\epsilon}(x) - h_{k,\epsilon}(x')\big)^2\nonumber\\
&= \big(\widetilde{h}_{k,\epsilon}(\operatorname{proj}_{V_{k}^\perp}(x))-\widetilde{h}_{k,\epsilon}(\operatorname{proj}_{V_{k}^\perp}(x'))\big)^2\nonumber\\
& \le \|\operatorname{proj}_{V_{k}^\perp}(x) - \operatorname{proj}_{V_{k}^\perp}(x')\|_{\Omega_{k}}^2\nonumber\\
&= \scal{\operatorname{proj}_{V_{k}^\perp}(x-x')}{\Omega_{k}\operatorname{proj}_{V_{k}^\perp}(x-x')}\nonumber\\
&= \scal{\Omega_{k}^{1/2}\operatorname{proj}_{V_{k}^\perp}(x-x')}{\Omega_{k}^{1/2}\operatorname{proj}_{V_{k}^\perp}(x-x')}\nonumber\\
& = \scal{x-x'}{\Omega_{k}(x-x')}.
\end{align}
If $V_{k}^\perp = \{0\}$, then it follows from \eqref{e:pQLip} that $h_{k}=0$. Therefore, \eqref{e:approxhkhkeps} holds with $h_{k,\epsilon}=0$, which belongs to  $\mathcal{N}_{\mathcal{F}}(V_{k}^\perp,\RR)$ and obviously satisfies \eqref{e:Liphkeps}.

Condition \eqref{e:condpk} means that $(p_{k})_{1\le k \le K}$ is an orthornormal basis of $\RR^{K}$ in the standard Euclidean metric. This implies that
\begin{equation}
(\forall x \in \RR^K) \quad Q(x) = \sum_{k=1}^{K} h_{k}(x)\,p_{k}.
\end{equation}
Set 
\begin{equation}
(\forall x \in \RR^K) \quad Q_{\epsilon}(x) = \sum_{k=1}^{K} h_{k,\epsilon}(x)\,p_{k}.
\end{equation}
It follows from \eqref{e:Liphkeps} and \eqref{e:stat22} that, for every $(x,x')\in (\RR^{K})^{2}$,
\begin{align}
&\|Q_{\epsilon}(x) -Q_{\epsilon}(x')\|^{2}\nonumber\\
& = \sum_{k=1}^{K} \big(h_{k,\epsilon}(x) - h_{k,\epsilon}(x')\big)^2\nonumber\\
&\le  \sum_{k=1}^{K} \scal{x-x'}{\Omega_{k}(x-x')}\nonumber\\
& \le \|x-x'\|^{2},
\end{align}
which shows that $Q_{\epsilon}\in \mathrm{Lip}_{1}(\RR^{K},\RR^{K})$. In addition since, for every $x\in \RR^{K}$, 
\begin{equation}
Q_{\epsilon}(x) = W [h_{1,\epsilon}(x),\ldots,h_{K,\epsilon}(x)]^\top
\end{equation}
with $W = [p_1,\ldots,p_{K}]$ and, for every $k\in \NN$, $h_{k,\epsilon} \in \mathcal{N}_{\mathcal{F}}(\RR^{K},\RR)$, $Q_{\epsilon}$ belongs to $\mathcal{N}_{\mathcal{F}}(\RR^{K},\RR^{K})$. Let $A_{\epsilon} = 2(\Id+Q_{\epsilon}\big)^{-1}-\Id$. We finally deduce from \eqref{e:approxhkhkeps} that, for every $x\in S$,
\begin{align}
&\|J_{A}(x)-J_{A_{\epsilon}}(x)\|^{2}\nonumber\\
&= \Big\|\frac{x+Q(x)}{2}-\frac{x+Q_{\epsilon}(x)}{2}\Big\|^{2}\nonumber\\
&= \frac{1}{4} \sum_{k=1}^{K} \big(h_{k}(x)-h_{k,\epsilon}(x)\big)^{2} \le \epsilon^{2}.
\end{align}

\noindent\ref{t:approxmonii}: Let $(x,y)\in (\RR^{K})^2$.
We have 
\begin{equation}
y \in A(x)  \quad \Leftrightarrow  \quad x = J_{A}(x+y).
\end{equation}
Assume that $x+y\in S$. It follows from \ref{t:approxmoni} that there exists $x_{\epsilon}\in \RR^{K}$ such that $x_{\epsilon} = J_{A_{\epsilon}}(x+y)$ and $\|x-x_{\epsilon}\| \le \epsilon$. Let $y_{\epsilon}= x-x_{\epsilon}+y$. We have $x_{\epsilon} = J_{A_{\epsilon}}(x_{\epsilon}+y_{\epsilon})$, that is $y_{\epsilon} \in A_{\epsilon}(x_{\epsilon})$. In addition, $\|y-y_{\epsilon}\| = \|x-x_{\epsilon}\|\le \epsilon$.
\end{proof}

We will next show that \cref{t:approxmon} extends to a wider class of MMOs.

\begin{corollary}
Let $\HH=\RR^K$. Let $(\omega_i)_{1\le i \le I} \in ]0,1]^I$ be such that \mbox{$\sum_{i=1}^I \omega_i=1$}. For every $i\in \{1,\ldots,I\}$, let $A_i\colon \HH \to 2^{\HH}$ be a stationary MMO. Then the same properties as in  \cref{t:approxmon} hold if $A\colon \HH \to 2^{\HH}$ is the MMO with resolvent \mbox{$J_A = \sum_{i=1}^I \omega_i J_{A_i}$}.
\end{corollary}
\begin{proof}
First note that $J_A\colon \HH\to \HH$ is firmly nonexpansive \cite[Proposition 4.6]{bauschke2017convex}), hence $A$ is indeed an MMO.
As a consequence of \cref{t:approxmon}, for every compact set $S \subset \HH$ and every $\epsilon \in \RPP$, there exist NNs $(Q_{i,\epsilon})_{1\le i\le I}$ in $\mathcal{N}_{\mathcal{F}}(\HH,\HH)$ such that $(A_{i,\epsilon})_{1\le i \le I} = \big(2(\Id+Q_{i,\epsilon}\big)^{-1}-\Id\big)_{1\le i \le I}$ satisfy:
\begin{equation} \label{e:approxJAi}
    (\forall i \in \{1,\ldots,Q\})(\forall x \in S)\quad 
    \|J_{A_i}(x)-J_{A_{i,\epsilon}}(x)\| \le \epsilon.
\end{equation}
Let $Q_\epsilon = \sum_{i=1}^I \omega_i Q_{i,\epsilon}$. Then $Q_\epsilon \in \mathrm{Lip}_{1}(\RR^{K},\RR^{K})$ and, since it is built from  a linear combination of the outputs of $I$ NNs in $\mathcal{N}_{\mathcal{F}}(\HH,\HH)$ driven with the same input, it belongs to
$\mathcal{N}_{\mathcal{F}}(\HH,\HH)$. In addition, $A_\epsilon = 2(\Id+Q_{\epsilon}\big)^{-1}-\Id$ is such that
\begin{equation}
J_{A_\epsilon} =
    \frac12 \Big(\sum_{i=1}^I \omega_i Q_{i,\epsilon}+\Id\Big) = 
    \sum_{i=1}^I \omega_i J_{A_{i,\epsilon}},
\end{equation}
which allows us to deduce from \eqref{e:approxJAi} that
\begin{equation}
   (\forall x \in S)\quad 
    \|J_{A}(x)-J_{A_{\epsilon}}(x)\| \le 
    \sum_{i=1}^I \omega_i
    \|J_{A_i}(x)-J_{A_{i,\epsilon}}(x)\| \le \epsilon.
\end{equation}
The rest of the proof follows the same line as for \cref{t:approxmon}.
\end{proof}

\begin{remark}
The above results are less accurate than standard universal approximations ones which, for example, guarantee an arbitrary close approximation to any continuous function with a network having only one hidden layer \cite{Horn89,Lesh93}. Indeed, the requirement that the resolvent of a MMO must be firmly nonexpansive induces some significant increase of the difficulty of the mathematical problem. Nonetheless, the firm nonexpansiveness will enable us to build convergent PnP algorithms described in the next sections.
\end{remark}

\section{Proposed algorithm}
\label{sec:3}

\subsection{Forward-backward algorithm} 

Let us now come back to problems of the form \eqref{e:mon}. Such monotone inclusion problems can be tackled by a number of algorithms \cite{combettes2018monotone,combettes2020fixed}, which are all grounded on the use of the resolvent of $A$ (or a scaled version of this operator).
For simplicity, let us assume that $f$ is a smooth function. In this case, a famous algorithm for solving \eqref{e:mon} is the forward-backward (FB) algorithm \cite{chen97, combettes05}, which is expressed as
\begin{equation}\label{e:algFB}
(\forall n \in \NN)\quad 
x_{n+1}=J_{\gamma A}\big(x_{n}-\gamma \nabla f(x_{n})\big)
\end{equation}
where $\gamma >0$. If a neural network $\net$ is used to approximate $J_{\gamma A}$, then a natural substitute for \eqref{e:algFB} is 
\begin{equation}\label{e:algFBpap}
(\forall n \in \NN)\quad
x_{n+1}=\net\big(x_{n}-\gamma \nabla f(x_{n})\big).
\end{equation}
The following convergence result then straightforwardly follows from standard asymptotic properties of the FB algorithm \cite{combettes05}.

\begin{proposition}\label{prop:convFB}
Let  $\mu \in \RPP$ and let $\gamma \in ]0,2/\mu[$. Let $f\colon \HH \to \RR$ be a convex differentiable function with $\mu$-Lipschitzian gradient. Let $\net$ be a neural network such that $\widetilde{J}$ is  $1/2$-averaged as in \eqref{e:JAQ}.
Let $\widetilde{A}$ be the maximally monotone operator equal to $(\widetilde{J}^{-1}-\Id)$. Assume that the set $\mathcal{S}_{\gamma}$ of zeros of $\nabla f + \gamma^{-1}\widetilde{A}$ is nonempty. Then,  the sequence $(x_{n})_{n\in \NN}$ generated by  iteration \eqref{e:algFBpap} converges (weakly) to $\widehat{x}\in \mathcal{S}_{\gamma}$, i.e., $\widehat{x}$ satisfies 
\begin{equation}\label{e:mon_learned}
0 \in \nabla f(\widehat{x}) + \gamma^{-1}\widetilde{A} (\widehat{x}).
\end{equation}
\end{proposition}

\begin{remark}\
\begin{itemize}
\item[(i)] For a given choice of $\widetilde{A}$, the stepsize $\gamma$ in the proposed PnP-FB does not only act on the convergence profile, but also on the set of solutions, as shown in \eqref{e:mon_learned}.
\item[(ii)] The technical assumption $\mathcal{S}_{\gamma} \neq \varnothing$ can be waived if $\nabla f+ \gamma^{-1}\widetilde{A}$ is strongly monotone \cite[Corollary 23.37]{bauschke2017convex}). Then there exists a unique solution to \eqref{e:mon_learned}.
This is achieved if $f$ is strongly convex or if $\net= (\Id+Q)/2$ where $Q=\widetilde{Q}/(1+\delta)$ with $\delta \in \RPP$ and $\widetilde{Q}\colon \HH\to \HH$ nonexpansive. In the latter case, $\widetilde{A} = B+ \delta(2+\delta)^{-1} \Id$ where $B\colon\HH\to 2^{\HH}$ is maximally monotone. An example of such a strongly monotone operator $\widetilde{A}$ is encountered in elastic net regularization.
\item[(iii)] A classical result by Rockafellar states that $\widetilde{A}$ is the subdifferential of some convex lower-semincontinuous function if and only if $\widetilde{A}$ is maximally \emph{cyclically} monotone \cite{combettes2018monotone,rockafellar1966characterization}. As we only enforce the maximal monotonicity, \eqref{e:mon_learned} does not necessarily correspond to a minimization problem in general.
\end{itemize}
\end{remark}

Algorithm \eqref{e:algFBpap} is common to other frameworks, such as regularization-by-denoising (RED) and unfolded networks. As underlined by the authors of \cite{cohen2020regularization}, the RED-PRO algorithm \cite[Algorithm 4.1]{cohen2020regularization} can be written as a PnP-FB algorithm of the same form as Algorithm~\eqref{e:algFBpap}, where $\net = (1-\alpha) \operatorname{Id} + \alpha \netav$, with $\alpha \in ]0,1/2[$ and $\netav$ being $d$-demicontractive for $d\in [0,1[$. 
So, although \cite[Algorithm 4.1]{cohen2020regularization} and Algorithm~\eqref{e:algFBpap} have the same structure (considering $\alpha=1/2$ and that nonexpansiveness implies $0$-demicontractive), they also exhibit several differences. 
First, in \cite{cohen2020regularization} the authors train $\netav$ as a denoiser, while we train $\widetilde{J}$. In practice, our approach is closer to classical PnP algorithms where a proximity operator is replaced by a denoiser (and not a relaxed version of the denoiser). Second, the characterization of the limit point differs since, in \cite[Theorems 4.3 \& 4.4]{cohen2020regularization}, it is shown that RED-PRO converges to a minimizer of $f$, where the solution is a fixed point of $\net$.

If a finite number $N$ of iterations of Algorithm~\eqref{e:algFBpap} are performed, unfolding the FB algorithm results in the NN architecture given in \cref{fig:NNFB}. 
\begin{figure}
\centering
\includegraphics[width=0.95\textwidth]{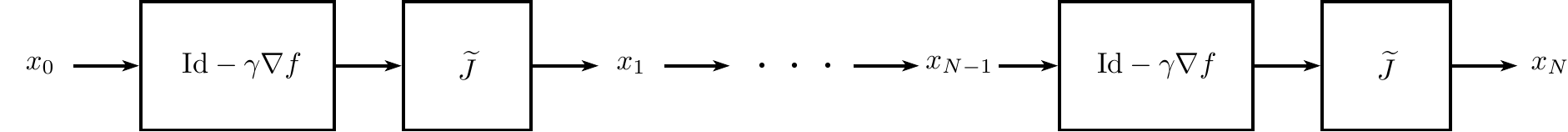}
\caption{\small Unfolded FB algorithm over $N$ iterations.}%
\label{fig:NNFB}
\end{figure}
If $\gamma < 2/\mu$, the gradient operator $(\Id-\gamma \nabla f)$ is a $\gamma\mu/2$-averaged operator. It can thus be interpreted as an activation operator \cite{combettes2019}. This activation operator is however non standard both because of its form and its dependence on the observed data $z$. A special case arises when $f$ corresponds to a least squares data fit term, i.e.,
\begin{equation}\label{e:LScrit}
(\forall x \in \HH) \quad f(x) = \frac{1}{2}\|Hx-z\|^{2},
\end{equation}
where $z$ belongs to some real Hilbert space $\GG$ and $H$ is a bounded operator from $\HH$ to $\GG$  modelling some underlying linear observation process (e.g. a degradation operator in image recovery). Then, $\nabla f\colon x \mapsto H^{*}(Hx-z)$ where $H^*$ denotes the adjoint of $H$ and $\mu = \|H\|^2$. Hence, $\Id-\gamma \nabla f$ is an affine operator involving a self-adjoint weight operator $\Id-\gamma H^{*}H$ and a bias $\gamma H^{*}z$. The unfolded network has thus a structure similar to a residual network where groups of layers are identically repeated and the bias introduced in the gradient operator depends on $z$. A parallel could also be drawn with a recurrent neural network driven with a stationary input, which would here correspond to $z$.
It is worth pointing out that, under the assumptions of \cref{prop:convFB}, the unfolded network in \cref{fig:NNFB} is robust to adversarial input perturbations, since it is globally nonexpansive. Note finally that, in the case when $f$ is given by \eqref{e:LScrit}, allowing the parameter $\gamma$ and the operator $\net$ to be dependent on $n \in \{1, \ldots,N\}$ in \cref{fig:NNFB} would yield an extension of ISTA-net \cite{zhang2018ista}. However, as shown in \cite{chen2018theoretical}, convergence of such a scheme requires specific assumptions on the target signal model. Other works have also proposed NN architectures inspired from primal dual algorithms \cite{adler2018learned,banert2020data,jiu2020deep}.

\subsection{Training}

A standard way of training a NN operating on $\HH=\RR^K$ for PnP algorithms is to train a denoiser for data corrupted with Gaussian noise \cite{zhang2018learning}. Let $\overline{\boldsymbol{x}} = (\overline{x}_{\ell})_{1 \le \ell \le L}$ be training set of $L$ images of $\HH$ and let
\begin{equation}\label{eq:denoising}
    (\forall \ell\in \{1, \ldots, L\}) \qquad y_{\ell} = \overline{x}_{\ell}+\sigma_{\ell} w_{\ell}
\end{equation}
be a noisy observation of $\overline{x}_{\ell}$, where $\sigma_{\ell}\in \RPP$ and $(w_{\ell})_{1\le {\ell} \le L}$ are assumed to be realizations of standard normal i.i.d. random variables. In practice, either $\sigma_{\ell} \equiv \sigma >0$ is chosen to be constant during training \cite{zhang2017learning}, or $\sigma_{\ell}$ is chosen to be a realization of a random variable with uniform distribution in $[0,\sigma]$,  for $\sigma\in \RPP$ \cite{zhang2019deep}.

The NN $\net$ described in the previous section will be optimally chosen within a family $\{\net_\theta \mid \theta\in \RR^P\}$ of NNs. For example, the parameter vector $\theta$ will account for the convolutional kernels and biases of a given network architecture. An optimal value $\widehat{\theta}$ of the parameter vector is thus a solution to the following problem:
\begin{equation}
\minimize{\theta} 
\sum_{\ell=1}^L
\|\net_\theta(y_{\ell})-\overline{x}_{\ell}\|^2
\quad \text{s.t.} \quad
\netav_\theta = 2\net_\theta-\Id \,\,\text{ is nonexpansive}.
\label{eq:simplified}
\end{equation}
(The squared $\ell_2$ norm in \eqref{eq:simplified} can be replaced by another cost function, e.g., an $\ell_1$ norm \cite{zhang2020plug}.)
The main difficulty with respect to a standard training procedure is the nonexpansiveness constraint stemming from \cref{prop:basmaxmon} which is crucial to ensure the convergence of the overall PnP algorithm.
In this context, the tight sufficient conditions described in \cref{p:propnonexpan}  for building the associated nonexpansive operator $\netav_\theta$ are however difficult to enforce. For example, the maximum value of the left-hand side in inequality~\eqref{e:Lipboundsep} is NP-hard to compute \cite{virmaux2018lipschitz} and estimating an accurate estimate of the Lipschitz constant of a NN requires some additional assumptions \cite{neacsu2020accuracy} or some techniques which do not scale well to high-dimensional data \cite{fazlyab2019efficient}. In turn, by assuming that, for every $\theta \in  \RR^P$ $Q_\theta$ is differentiable, we leverage on the fact that $\netav_\theta$ is nonexpansive if and only if its Jacobian $\jac \netav_\theta$ satisfies
\begin{equation}\label{eq:constrained_jac}
(\forall x \in \HH)\qquad \| \jac \netav_\theta(x) \| \le 1. 
\end{equation}
In practice, one cannot enforce the constraint in \eqref{eq:constrained_jac} for all $x \in \HH$. 
We therefore propose to impose this constraint on every segment $[\overline{x}_{\ell},\net_\theta(y_{\ell})]$ with 
$\ell\in \{1, \ldots, L\}$, or more precisely at points
\begin{equation}
    \widetilde{x}_{\ell} = \varrho_{\ell} \overline{x}_{\ell}+(1-\varrho_{\ell})\net_{\theta}(y_{\ell}),
\label{eq:xtilde_def}
\end{equation} 
where $\varrho_{\ell} $ is a realization of a random variable with uniform distribution on [0,1].
To cope with the resulting constraints, instead of using projection techniques which might be slow \cite{terris2020building} and raise convergence issues when embedded in existing training algorithms \cite{alacaoglu2020convergence}, we propose to employ an exterior penalty approach.
The final optimization problem thus reads 
\begin{equation}    \label{pb:training_final}
    \minimize{\theta} 
    \sum_{\ell=1}^L \Phi_{\ell}(\theta),
\end{equation}
where, for every $\ell\in \{1,\ldots,L\}$,
\begin{equation} 
\begin{aligned}
    \Phi_{\ell}(\theta) 
    &= 
    \|\net_\theta(y_{\ell})-\overline{x}_{\ell} \|^2 
    + \lambda \max\left\{\|\jac \netav_\theta(\widetilde{x}_{\ell})\|^2, 1-\varepsilon\right\},
\end{aligned}
\label{eq:loss_summarized}
\end{equation} 
$\lambda\in \RPP$ is a penalization parameter, and $\varepsilon \in ]0,1[$ is a parameter allowing us to control the constraints. Standard results concerning penalization methods \cite[Section 13.1]{luenberger2016linear}, guarantee that, if
$\widehat{\theta}_\lambda$ is a solution to \eqref{pb:training_final} for $\lambda\in \RPP$, then 
$(\forall \ell \in \{1,\ldots,L\})$
$\lim_{\lambda \to +\infty} \|\jac \netav_{\widehat{\theta}_\lambda}(\Tilde{x}_{\ell})\|^2 \le 1-\varepsilon$.
Then, since $\varepsilon>0$, there exists $\overline{\lambda} \in \RPP$ such that,
for every $\lambda \in [\overline{\lambda},+\infty[$ and
every $\ell \in \{1,\ldots,L\}$, 
$\|\jac \netav_{\widehat{\theta}_\lambda}(\Tilde{x}_{\ell})\| \le 1$.

\begin{remark}
Hereabove, we have made the assumptions that the network is differentiable.
Automatic differentiation tools however are applicable to networks which contain nonsmooth linearities such as ReLU (see \cite{bolte2020conservative} for a theoretical justification for this fact).
\end{remark}

To solve \eqref{pb:training_final} numerically, we resort to the Adam optimizer \cite{zhang2019deep} as described in \cref{alg:training}. 
This algorithm uses a fixed number of iterations \mbox{$N\in\NN^*$} and relies on approximations to the gradient of $\sum_{\ell}\Phi_{\ell}$ computed on randomly sampled batches of size $D$, selected from the training set of images $(\overline{x}_{\ell})_{1 \le \ell \le L}$. More precisely, at each iteration $t \in \{1, \ldots, N\}$, we build the approximated gradient $ \frac{1}{D} \sum_{d=1}^D g_d$ (see lines 3-9), followed by an Adam update (line 10) consisting in a gradient step on $\theta_d$ with adaptive moment \cite{kingma2014adam}. 
Then the approximated gradient is computed as follows. For every $d \in \{1 , \ldots, D\}$, we select randomly an image from the training set (line 4), we draw at random a realization of a normal i.i.d. noise that we use to build a noisy observation $y_d$ (line 5-6). We then build $\widetilde{x}_d$ as in \eqref{eq:xtilde_def} (lines 5-7) and compute the gradient $g_d$ of the loss $\Phi_d$ w.r.t. to the parameter vector at its current estimate
$\theta_n$ (line 8). Note that any other gradient-based algorithm, such as SGD or RMSprop \cite{mukkamala2017variants} could be used to solve~\eqref{pb:training_final}. 

\begin{algorithm}
\caption{Adam algorithm to solve \eqref{pb:training_final}}
\label{alg:training}
\begin{algorithmic}[1]
\STATE{Let $D\in\NN^*$ be the batch size, and $N\in\NN^*$ be the number of training iterations.}
\FOR{$n=1,\ldots,N$}
\FOR{$d=1,\ldots,D$}
\STATE{Select randomly $\ell \in \{1, \ldots, L\}$};
\STATE{Draw at random $w_d \sim \mathcal{N}(0,1)$ and $\varrho_d\sim \mathcal{U}([0,1])$};
\STATE{$y_d = \overline{x}_{\ell}+\sigma w_d$};
\STATE{$\widetilde{x}_d = \varrho_d \overline{x}_{\ell}+(1-\varrho_d)\net_{\theta_n}(y_d)$};
\STATE{$g_d =\nabla_{\theta} \Phi_d(\theta_n)$};
\ENDFOR
\STATE{$\theta_{n+1} = \text{Adam}( \frac{1}{D}\sum_{d=1}^Dg_d,\theta_n)$};
\ENDFOR
\RETURN $\net_{\theta_{N}}$
\end{algorithmic}
\end{algorithm}

In order to compute the spectral norm $\| \jac \netav_\theta(x) \|$ for a given image $x \in \HH$, we use the power iterative method (see details in Appendix~\ref{apx:powit}). 

\begin{remark}\ \label{rk:nonexp}
\begin{enumerate}
\item \label{rk:nonexp:i} Other works in the GAN literature have investigated similar regularizations \cite{gulrajani2017improved, petzka2017regularization, wei2018improving} to constrain the gradient norm of the discriminator (recall that, the disciminator being a function with values in $\RR$, its gradient is well defined). In our work, we aim at constraining the Lipschitz constant of \mbox{$\netav:\RR^N\to\RR^N$}, the gradient of which is not defined (only its Jacobian is), hence the need of a more involved method for computing the spectral norm.
\item \label{rk:nonexp:ii}A related approach was presented in \cite{hoffman2019robust}, where the loss is regularized with the Froebenius norm of the Jacobian. The latter is not enough to ensure convergence of the PnP method \eqref{e:algFBpap} which requires to constrain the spectral norm $\|\cdot\|$ of the Jacobian.
\item \label{rk:nonexp:iii}The power iterative method has been used in previous works \cite{miyato2018spectral, ryu2019plug, yoshida2017spectral}. In particular, RealSN \cite{ryu2019plug} relies on a power iteration to compute the Lipschitz constant of each convolutional layer. In our case, we compute the spectral norm of the Jacobian for the full network $\netav$.
\end{enumerate}
\end{remark}

\section{Simulations and results}
\label{sec:4}

\subsection{Experimental setting}
\label{Ssect:exp_settings}

\paragraph{Inverse Problem}
We focus on inverse deblurring imaging problems, where the objective is to find an estimate $\widehat{x} \in \RR^K$ of an original unknown image $\overline{x} \in \RR^K$, from degraded measurements $z\in \RR^K$ given by
\begin{equation}
z = H \overline{x}+e, 
\label{eq:invpb}
\end{equation}
where $H \colon \RR^K\to\RR^K$ is a blur operator and $e \in \RR^K$ is a realization of an additive white Gaussian random noise with zero-mean and standard deviation $\nu\in \RPP$. In this context, a standard choice for the data-fidelity term is given by \eqref{e:LScrit}. In our simulations, $H$ models a blurring operator implemented as a circular convolution with impulse response $h$. We will consider different kernels $h$ taken from \cite{levin2009understanding} and \cite{bertocchi2020deep}, see \cref{fig:kernels} for an illustration. The considered kernels are normalized such that the Lipschitz constant $\mu$ of the gradient of $f$ is equal to $1$.

\paragraph{Datasets} 
Our training dataset consists of $50000$ test images from the ImageNet dataset \cite{deng2009imagenet} that we randomly split in $98\%$ for training and $2\%$ for validation. In the case of grayscale images, we investigate the behavior of our method either on the full BSD68 dataset \cite{martin2001BSDS} or on a subset of 10 images, which we refer to as the BSD10 set. For color images, we consider both the BSD500 test set~\cite{martin2001BSDS} and the Flickr30 test set~\cite{xu2014deep}.\footnote{We consider normalised images, where the coefficient values are in $[0,1]$ (resp. $[0,1]^3$) for grayscale (resp. color) images.} Eventually, when some fine-tuning is required, we employ the Set12 and Set18 datasets \cite{zhang2017learning} for grayscale and color images, respectively.

\begin{figure}[t]
\centering
\includegraphics[width=0.9\textwidth]{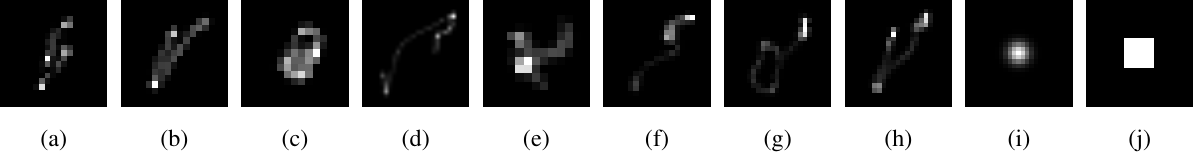}
\caption{\small Blur kernels used in our simulations. (a)-(h) are kernels 1-8 from \cite{levin2009understanding} respectively while (i) is the kernel from the GaussianA setup and (j) from the Square setup in \cite{bertocchi2020deep}.}
\label{fig:kernels}
\end{figure}

\paragraph{Network architecture and pretraining} 
In existing PnP algorithms involving NNs (see e.g. \cite{ledig2017photo, zhang2017beyond, zhang2017learning, zhang2019deep}), the NN architecture $\net$ often relies on residual skip connections. This is equivalent, in \eqref{e:defNN}, to set $Q =  \Id + \widetilde{T}_{M} \ldots \widetilde{T}_1$ where, for every $m\in \{1,\ldots,M\}$, $\widetilde{T}_m$ is standard neural network layer (affine operator followed by activation operator).
More specifically, the architecture we consider for $\net$ is such that $M=20$. It is derived from the DnCNN-B architecture \cite{zhang2017beyond} from which we have removed batch normalization layers and where we have replaced ReLUs with LeakyReLUs (see \cref{fig:arch_dncnn}). 

We first pretrain the model $\net$ in order to perform a denoising task for a variable level of noise, without any Jacobian regularization. For each training batch, we generate randomly sampled patches of size $50\times 50$ from images that are randomly rescaled and flipped. More precisely, we consider Problem~\eqref{pb:training_final}-\eqref{eq:loss_summarized} with $\lambda=0$, and $(\sigma_\ell)_{1 \le \ell \le L}$ chosen to be realizations of i.i.d. random variable with uniform distribution in $[0,0.1]$ for each patch. 
We use the Adam optimizer \cite{kingma2014adam} to pretrain the network with learning rate $10^{-4}$, and considering $150$ epochs, each consisting of $490$ iterations of the optimizer. The learning rate is divided by $10$ after $100$ epochs and we clip gradient norms at $10^{-2}$ for an improved stability during training. 
This pretrained network will serve as a basis for our subsequent studies. The details regarding the training of our networks will be given on a case-by-case basis in the following sections. 

All models are trained on 2 Nvidia Tesla 32 Gb V100 GPUs and experiments are performed in PyTorch\footnote{Code publicly available at  \url{https://github.com/basp-group/xxx} (upon acceptance of the paper)}.

\begin{figure}
    \centering
    \includegraphics[width=0.8\textwidth]{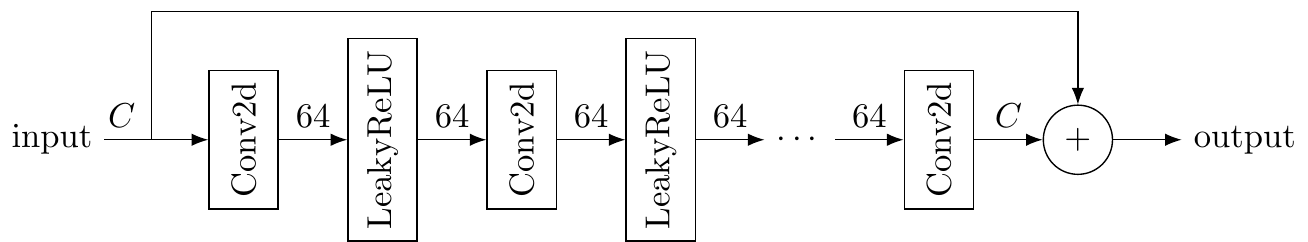}
    \caption{\small Proposed DnCNN architecture of $\net$, with a total of 20 convolutional layers. It corresponds to a modified version of the DnCNN-B architecture \cite{zhang2017beyond}.
    The number of channels $C$ is indicated above arrows ($C=1$ for grayscale images and $C=3$ for color ones).}
    \label{fig:arch_dncnn}
\end{figure}

\paragraph{Goal} 
We aim to study the PnP-FB algorithm~\eqref{e:algFBpap} where $\net$, chosen according to the architecture given in \cref{fig:arch_dncnn}, has been trained in order to solve~\eqref{eq:loss_summarized}. 
We will first study the impact of the choice of the different parameters appearing in the training loss~\eqref{eq:loss_summarized} on the convergence of the PnP-FB algorithm and on the reconstruction quality. Then, we will compare the proposed method to state-of-the-art iterative algorithms either based on purely variational or PnP methods. 

We evaluate the reconstruction quality with Peak Signal to Noise Ratio (PSNR) and Structural Similarity Index Measure (SSIM) metrics \cite{wang2004image}. The PSNR between an image $x\in \RR^K$ and the ground truth $\overline{x}\in \RR^K$ is defined as
\begin{equation}
    \operatorname{PSNR}(x,\overline{x}) = 20\log_{10}\left(\frac{\sqrt{K}\,\max_{1 \le \ell \le L} \overline{x}_{\ell}}{\|x-\overline{x}_{\ell}\|}\right),
\end{equation}
where, in our case, we have $\max_{1 \le \ell \le L} \overline{x}_\ell = 1$. The SSIM is given by
\begin{equation}
    \operatorname{SSIM}(x,\overline{x}) =\frac{(2\mu_x\mu_{\overline{x}}+\vartheta_1)(2\sigma_{x\overline{x}}+\vartheta_2)}{(\mu_x^2+\mu_{\overline{x}}^2+\vartheta_1)(\sigma_x^2+\sigma_{\overline{x}}^2+\vartheta_2)},
\end{equation}
where $(\mu_x,\sigma_x)$ and $(\mu_{\overline{x}},\sigma_{\overline{x}})$ are the mean and the variance of $x$ and $\overline{x}$ respectively, $\sigma_{x\overline{x}}$ is the cross-covariance between $x$  and $\overline{x}$, and $(\vartheta_1,\vartheta_2) =(10^{-4},9\times 10^{-4})$.

\subsection{Choice of the parameters}
\label{Ssec:lipreg} 

In this section, we study the influence of the parameters\footnote{We noticed that the influence of $\varepsilon>0$ in \eqref{eq:loss_summarized} was negligible compared with that of $(\lambda, \sigma, \gamma)$.} $(\lambda, \sigma, \gamma)$ on the results of the PnP-FB algorithm~\eqref{e:algFBpap} applied to the NN in \cref{fig:arch_dncnn}. We recall that $\lambda$ is the parameter acting on the Jacobian regularization, $\sigma$ is the noise level for which the denoiser is trained, and $\gamma$ is the stepsize in the PnP-FB algorithm~\eqref{e:algFBpap}.

\paragraph{Simulation settings}
We consider problem~\eqref{eq:invpb} with $H$ associated with the kernels shown in \cref{fig:kernels}(a)-(h), and $\nu=0.01$. In this section, we consider the grayscale images from the BSD68 dataset. 

To investigate the convergence behavior of the PnP-FB algorithm, we consider the quantity defined at iteration $n \in \NN\setminus\{0\}$ as
\begin{equation}    \label{eq:crit_cn}
c_n = \|x_{n}-x_{n-1}\|/\|x_0\|, 
\end{equation}
where $(x_n)_{n\in \NN}$ is the sequence generated by the PnP-FB algorithm \eqref{e:algFBpap}. Note that the quantity $(c_n)_{n\in \NN}$ is known to be monotonically decreasing if the network $\net$ is firmly nonexpansive \cite{bauschke2017convex}.

\paragraph{Influence of the Jacobian penalization}

We study the influence of $\lambda$ on the convergence behavior of the PnP-FB algorithm~\eqref{e:algFBpap}. In particular we consider $\lambda \in \{5\times 10^{-7}, 10^{-6}, 2\times 10^{-6}, 5\times 10^{-6}, 10^{-5}, 2\times 10^{-5}, 4\times 10^{-5}, 1.6\times 10^{-4}, 3.2 \times 10^{-4}, 6.4 \times 10^{-4}\}$.

After pretraining, we train our DnCNN by considering the loss given in~\eqref{eq:loss_summarized}, in which we set  $\varepsilon=5\times 10^{-2}$ and $\sigma=0.01$. The batches are built as in the pretraining setting. The network is trained for $100$ epochs and the learning rate is divided by $10$ at epoch $80$. The training is performed with \cref{alg:training} where $D=100$ and  $N = 4.9\times 10^4$. For Adam's parameters, we set the learning rate to $10^{-4}$ and the remaining parameters to the default values provided in \cite{kingma2014adam}. 

To verify that our training loss enables the firm nonexpansiveness of our NN $\net$, we evaluate the norm of the Jacobian $\|\jac \netav(y_{\ell}) \|$ on a set of noisy images $(y_{\ell})_{1 \le \ell \le 68}$, obtained from the BSD68 test set considering the denoising problem~\eqref{eq:denoising}. The maximum of these values is given in \cref{tab:lambdas} for the different considered values of $\lambda$. We observe that the norm of the Jacobian decreases as $\lambda$ increases and is smaller than $1$ for $\lambda\geq 10^{-5}$. 

We now investigate the convergence behavior of the PnP-FB algorithm, depending on $\lambda$, considering BSD10 (a subset of BSD68). In our simulations, we set $\gamma=1/\mu=1$.
In \cref{fig:ex2} we show the values $(c_n)_{1 \le n \le 1000}$ for $1000$ iterations, considering kernel (a) from \cref{fig:kernels} for the different values of $\lambda$. 
The case $\lambda=0$ corresponds to training a DnCNN without the Jacobian regularization.
We observe that the stability of the PnP-FB algorithm greatly improves as $\lambda$ increases: for $\lambda\geq 10^{-5}$, all curves are monotonically decreasing. These observations are in line with the metrics from \cref{tab:lambdas} showing that $\|\jac \netav(y_{\ell})\|\leq 1$ for $\lambda \geq 10^{-5}$. The case when $\|\jac \netav\|$ is slightly larger than $1$ is of interest. In particular, when $\lambda=5\times 10^{-6}$, $\|\jac \netav \|^2 \approx 1.03$ (see \cref{tab:lambdas}) one observes that most curves from \cref{fig:ex2} (e) are monotonically decreasing, but some are not. When $\lambda=10^{-6}$, $\|\jac \netav \|^2 \approx 1.35$ (see \cref{tab:lambdas}), none of the curves from \cref{fig:ex2} (c) are monotonically decreasing, but we observed convergence in some case for a larger number of iterations.

These results confirm that by choosing an appropriate value of $\lambda$, one can ensure $Q$ to be 1-Lipschitz, i.e. $\net$ to be firmly nonexpansive, and consequently we secure the convergence of the PnP-FB algorithm~\eqref{e:algFBpap}. It also confirms that $\|\jac \netav \| \leq 1$ is a sufficient condition to ensure convergence. Finally, we also emphasize that the average PSNR values obtained with the PnP-FB algorithm, for the different considered $\lambda$, has a bell shape with maximum obtained when $\lambda=10^{-5}$.

\begin{table}[h!]
\footnotesize
\setlength\tabcolsep{2pt}
\centering
\begin{tabular}{lccccc|ccccc}
\hline
$\lambda$ & $0$ & $5\!\times\!10^{-7}$ & $1\!\times\!10^{-6}$ & $2\!\times\!10^{-6}$ & $5\!\times\!10^{-6}$ & $1\!\times\!10^{-5}$ & $4\!\times\!10^{-5}$ & $1.6\!\times\!10^{-4}$ & $3.2\!\times\!10^{-4}$ \\
\hline
$\|\jac \netav\|^2$ & 31.36 & 1.65 & 1.349 & 1.156 & 1.028 & 0.9799 & 0.9449 & 0.9440 & 0.9401\\
\hline 
\end{tabular}
\caption{\small Numerical evaluation of the firm nonexpansiveness $\net$ on a denoising problem on the BSD68 test set for different values of $\lambda$. 
}
\label{tab:lambdas}
\end{table}

\begin{figure}[ht]
\centering
\includegraphics[width=1.0\textwidth]{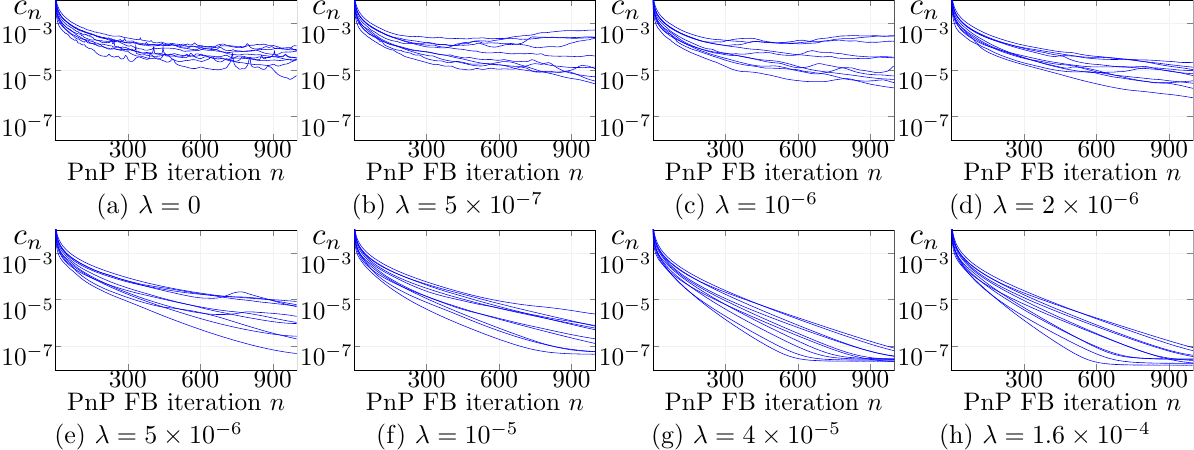}
\caption{\small Influence of $\lambda \in \{0,5\times 10^{-7}, 10^{-6}, 2\times 10^{-6},  5\times 10^{-6}, 10^{-5}, 4\times 10^{-5}, 1.6\times 10^{-4}\}$ on the stability of the PnP-FB algorithm for the deblurring problem with kernel in \cref{fig:kernels}(a). (a)-(h): On each graph, evolution of the quantity $c_n$ defined in~\eqref{eq:crit_cn} for each image of the BSD10 test set,  for a value of $\lambda$,  along the iterations of the PnP-FB algorithm \eqref{e:algFBpap}.}
\label{fig:ex2}
\end{figure}

\paragraph{Influence of the stepsize and training noise level}

Second, we investigate the influence $(\sigma, \gamma)$ on the reconstruction quality of the images restored with the PnP-FB algorithm. Indeed, according to \cref{prop:convFB}, the value of $\gamma$ acts on the solution, and hence potentially on the reconstruction quality.
We train the NN $\net$ given in \cref{fig:arch_dncnn} for $\sigma  \in \{0.005, 0.006, 0.007, 0.008, 0.009,\allowbreak 0.01\}$. As per the procedure followed in the study of the parameter $\lambda$, after pretraining, we train  $\net$  by considering the loss given in~\eqref{eq:loss_summarized}, in which we set  $\varepsilon=5\times 10^{-2}$ and $\lambda=10^{-5}$.
The batches are built as in the pretraining setting. The network is trained for $100$ epochs and the learning rate is divided by $10$ at epoch $80$. The training is performed with \cref{alg:training} where $D=100$ and  $N = 4.9\times 10^4$. For Adam's parameters, we set the learning rate to $10^{-4}$ and the remaining parameters to the default values provided in \cite{kingma2014adam}. We subsequently plug the trained DnCNN $\net$ in the PnP-FB algorithm \eqref{e:algFBpap}, considering different values for $\gamma \in  ]0,2[$. 
In these simulations, we focus on the case when the blur kernel in Problem~\eqref{eq:invpb} corresponds to the one shown in \cref{fig:kernels}(a).

Before discussing the simulation results, we present a heuristic argument suggesting that (i) $\sigma$ should scale linearly with $\gamma$, and (ii) the appropriate scaling coefficient is given as $2 \nu \|h\|$.
Given the choice of the data-fidelity term \eqref{e:LScrit}, the PnP-FB algorithm~\eqref{e:algFBpap} reads
\begin{align}\label{eq:iter_nongauss}
(\forall n \in \NN)\quad 
    x_{n+1} 
    &
    = \net\big( x_n - \gamma H^* (Hx_n - H \overline{x}-e) \big).
\end{align}
We know that, under suitable conditions, the sequence $(x_n)_{n\in \NN}$ generated by \eqref{eq:iter_nongauss} converges to a fixed point $\widehat{x}$, solution to the variational inclusion problem~\eqref{e:mon_learned}. 
We assume that $\widehat{x}$ lies close to $\overline{x}$ up to a random residual $e'=H(\widehat{x} - \overline{x})$, whose components are uncorrelated and with equal standard deviation, typically expected to be bounded from above by the standard deviation $\nu$ of the components of the original noise $e$. 
Around convergence, \eqref{eq:iter_nongauss} therefore reads as
\begin{align}\label{eq:iter_nongauss_fp}
   \widehat{x} = \net \left( \widehat{x} - \gamma H^*\left(e'-e\right)  \right),
\end{align}
suggesting that, $\net$ is acting as a denoiser of $\widehat{x}$ for an effective noise $-\gamma H^*(e'-e)$.
If the components of $e'-e$ are uncorrelated, the standard deviation of this noise is bounded by $\gamma\nu_{\text{eff}}$, with $\nu_{\text{eff}} =  2 \nu \|h\|$, a value reached when $e'=-e$. 
This linear function of $\gamma$ with scaling coefficient $\nu_{\text{eff}}$ thus provides a strong heuristic for the choice of the standard deviation $\sigma$ of the training noise.
For the considered kernel (shown in \cref{fig:kernels}(a)), we have $\nu_{\text{eff}}=0.0045$, so the interval $\sigma\in[0.005,0.01]$ also reads $\sigma\in[1.1\,  \nu_{\text{eff}}, 2.2\,\nu_{\text{eff}}]$.

In \cref{fig:metrics_gamma_sigma} we provide the average PSNR (left) and SSIM (right) values associated with the solutions to the deblurring problem for the considered simulations as a function of $\sigma/\gamma\nu_{\text{eff}}$. 
For each sub-figure, the different curves correspond to different values of $\gamma$. We observe that, whichever the values of $\gamma$, the reconstruction quality is sharply peaked around values of $\sigma/\gamma\nu_{\text{eff}}$ consistently around 1, thus supporting our heuristic argument. We also observe that  the peak value increases with $\gamma$. We recall that, according to the conditions imposed on $\gamma$ in \cref{prop:convFB} to guarantee theoretically the convergence of the sequence generated by PnP-FB algorithm, one has $\gamma < 2$. 
The values $\gamma=1.99$ and $\sigma/\gamma\nu_{\text{eff}}= 1$ (resp. $\gamma=1.99$ and $\sigma/\gamma\nu_{\text{eff}}= 0.9$) gives the best results for the PSNR (resp. SSIM). 

In \cref{fig:visual_gamma_sigma} we provide visual results for an image from the BSD10 test set, to the deblurring problem for different values of $\gamma$ and $\sigma$. The original unknown image $\overline{x}$ and the observed blurred noisy image are displayed in \cref{fig:visual_gamma_sigma}(a) and (g), respectively.
On the top row, we set $\sigma = 2 \nu_{\text{eff}}$, while the value of $\gamma$ varies from $1$ to $1.99$. We observe that the reconstruction quality improves when $\gamma$ increases, bringing the ratio $\sigma/\gamma\nu_{\text{eff}}$ closer to unity. Precisely, in addition to the PSNR and SSIM values increasing with $\gamma$, we can see that the reconstructed image progressively loses its oversmoothed aspect, showing more details. The best reconstruction for this row is given in \cref{fig:visual_gamma_sigma}(f), for $\gamma=1.99$.
On the bottom row, we set $\gamma=1$ and vary $\sigma$ from $1.3\, \nu_{\text{eff}}$ to $2.2\, \nu_{\text{eff}}$.
We see that sharper details appear in the reconstructed image when $\sigma$ decreases, again bringing the ratio $\sigma/\gamma\nu_{\text{eff}}$ closer to unity. The best reconstructions for this row are given in \cref{fig:visual_gamma_sigma}(h) and (i), corresponding to the cases $\sigma=1.3\, \nu_{\text{eff}}$ and $\sigma=1.6\, \nu_{\text{eff}}$, respectively.
Overall, as we have already noticed, the best reconstruction is obtained for  $\gamma=1.99$ and $\sigma/\gamma \nu_{\text{eff}}=1$, for which the associated image is displayed in \cref{fig:visual_gamma_sigma}(f). 
It can also be noticed that for all the considered values of $\sigma$, the best quality reconstructions are achieved when $\gamma=1.99/\mu$.
These results further support both our analysis of \cref{fig:metrics_gamma_sigma} and our heuristic argument for a linear scaling of $\sigma$ with $\gamma$, with scaling coefficient closely driven by the value $\nu_{\text{eff}}$.
The best strategy relies on choosing $\gamma$ the largest as possible, and deduce the value of $\sigma$ as $\sigma = \gamma \nu_{\text{eff}}$. In other words, our results suggest that the optimal values for both  the stepsize and training noise level can be determined as functions of parameters of the data fidelity term $f$, namely the Lipschitz constant $\mu$ of its gradient and the norm of the blurring kernel $h$, with $\gamma=1.99/\mu$ and $\sigma = 4 \nu \|h\|/\mu$.

\begin{figure}[t]
\centering
\includegraphics[width=.8\textwidth]{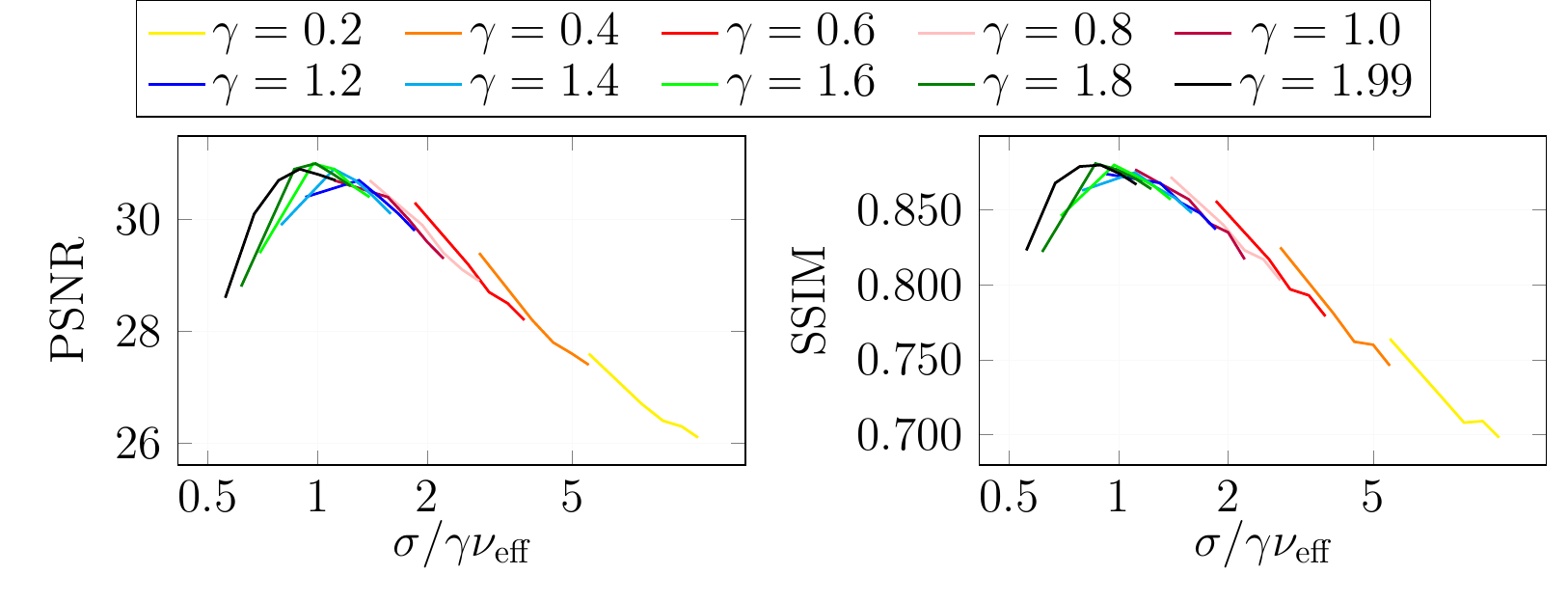}
\caption{\small  Influence of $\gamma\in]0,1.99]$ and $\sigma\in[0.005,0.01]$ on the reconstruction quality for the deblurring problem with kernel from \cref{fig:kernels}(a) on the BSD10 test set. For this experiment $\nu_{\text{eff}}=0.0045$. Left: average PSNR, right: average SSIM.}
\label{fig:metrics_gamma_sigma}
\end{figure}

\begin{figure}[t]
\centering
\includegraphics[width=.99\textwidth]{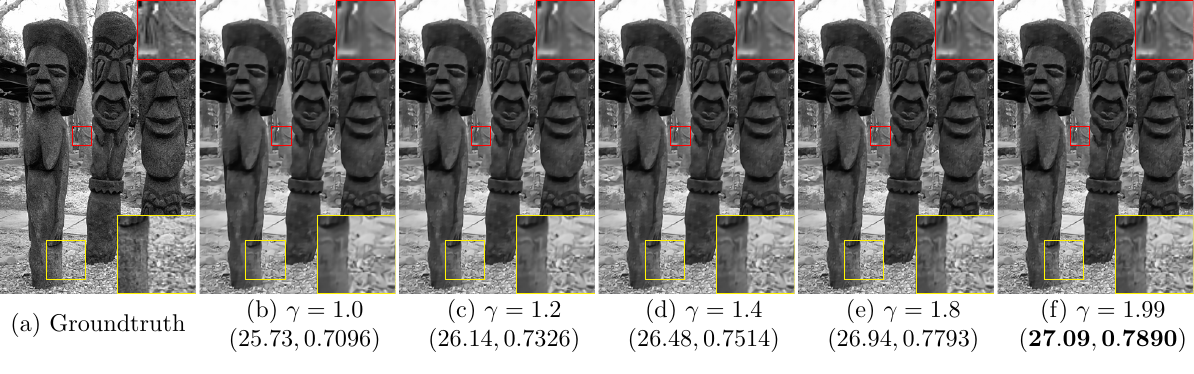}
\includegraphics[width=.99\textwidth]{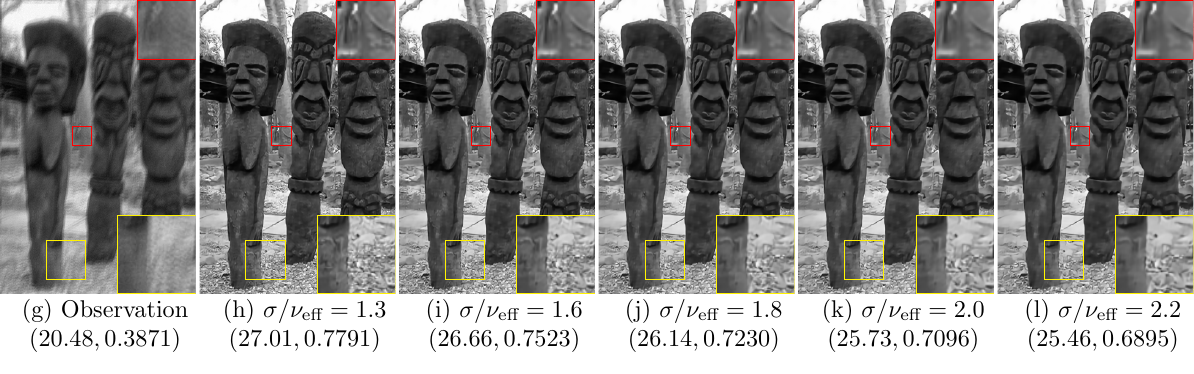}

\vspace{-0.2cm}

\caption{\small  Reconstructions of an image from the BSD10 test set obtained with the PnP-FB algorithm~\eqref{e:algFBpap} for the deblurring problem with kernel from \cref{fig:kernels}(a) for which  $\nu_{\text{eff}}=0.0045$. Top row: results for $\gamma\in [1,1.99]$ in algorithm \eqref{e:algFBpap} and $\sigma/\nu_{\text{eff}}=2$ (i.e. $\sigma=0.009$). Bottom row: results for $\sigma/\nu_{\text{eff}} \in [1.3,2.2]$ during training in \eqref{eq:loss_summarized} and $\gamma=1$ in algorithm~\eqref{e:algFBpap}.}
\label{fig:visual_gamma_sigma}
\end{figure}

\subsection{Comparison with other PnP methods}

In this section we investigate the behavior of the PnP-FB algorithm~\eqref{e:algFBpap} with $\net$ corresponding either to the proposed DnCNN provided in \cref{fig:arch_dncnn}, or to other denoisers. In this section, we aim to solve problem~\eqref{eq:invpb}, considering either grayscale or color images.

\paragraph{Grayscale images} 

We consider the deblurring problem~\eqref{eq:invpb} with $H$ associated with the kernels from \cref{fig:kernels}(a)-(h), $\nu=0.01$, evaluated on the BSD10 test set. 

We choose the parameters of our method to be the ones leading to the best PSNR values in \cref{fig:metrics_gamma_sigma}, i.e. $\sigma=0.009$ and $\gamma = 1.99$ corresponding to $\sigma/\gamma\nu_{\text{eff}}= 1$ for the kernel (a) of \cref{fig:kernels}, and we set $\lambda= 10^{-5}$.

We compare our method with other PnP-FB algorithms, where the denoiser corresponds either to RealSN \cite{ryu2019plug}, BM3D \cite{makinen2020collaborativeBM3D}, DnCNN \cite{zhang2017beyond}, or standard proximity operators \cite{combettes05, moreau1965proximite}. In our simulations, we consider the proximal operators of the two following functions: (i) the $\ell_1$-norm composed with a sparsifying operator consisting in the concatenation of the first eight Daubechies (db) wavelet bases \cite{carrillo2013sparsity, mallat2008wavelet}, and (ii) the total variation (TV) norm \cite{rudin1992nonlinear}. In both cases, the regularization parameters are fine-tuned on the Set12 dataset \cite{zhang2017beyond} to maximize the reconstruction quality. Note that the training process for RealSN has been adapted for the problem of interest.

\begin{figure}[t]
\centering
\includegraphics[width=1.0\textwidth]{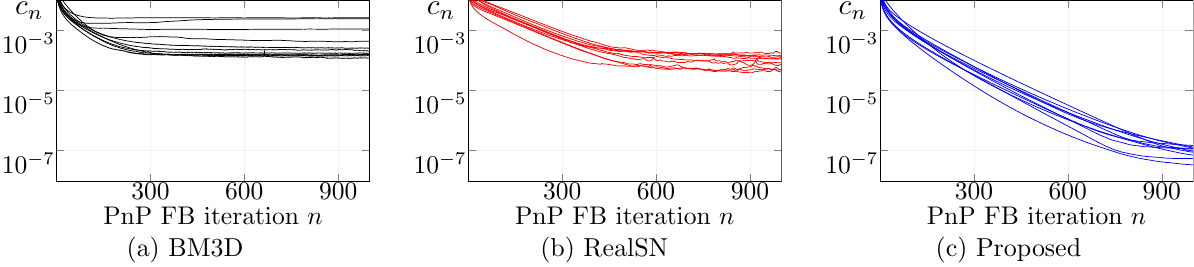}

\vspace{-0.2cm}

\caption{\small 
Convergence profile of the PnP-FB algorithm \eqref{e:algFBpap} for different denoisers plugged in as $\net$, namely BM3D (a), RealSN (b) and the proposed firmly nonexpansive DnCNN (c). Results are shown for the deblurring problem~\eqref{eq:invpb} with kernel from \cref{fig:kernels}(a). Each graph shows the evolution of $c_n$ defined in~\eqref{eq:crit_cn} for each image of the BSD10 test set.}
\label{fig:cv_comp}
\end{figure}

\begin{table}[ht]
\footnotesize
\setlength\tabcolsep{4pt}
\centering
\begin{tabular}{llccccccccc}
\hline
\multicolumn{2}{c}{\multirow{2}{*}{denoiser}} &   \multicolumn{8}{c}{kernel (see \cref{fig:kernels})} & \multirow{2}{*}{convergence} \\
\cline{3-10}
  && (a) & (b) & (c) & (d) &  (e) &  (f) &  (g) &  (h)\\
\hline
\multicolumn{2}{c}{Observation} & $23.36$ & $22.93$ & $23.43$ & $19.49$ & $23.84$ & $19.85$ & $20.75$ & $20.67$ \\
\multicolumn{2}{c}{RealSN \cite{ryu2019plug}} & $26.24$ & $26.25$ & $26.34$ & $25.89$ & $25.08$ & $25.84$ & $24.81$ & $23.92$ & \checkmark \\
\multicolumn{2}{c}{$\text{prox}_{\mu_{\ell_1}\|\Psi^\dagger \cdot\|_1}$} & $29.44$ & $29.20$ & $29.31$ & $28.87$ & $30.90$ & $30.81$ & $29.40$ & $29.06$ & \checkmark \\
\multicolumn{2}{c}{$\text{prox}_{\mu_\text{TV} \|\cdot\|_\text{TV}}$} &  $29.70$ & $29.35$ & $29.43$ & $29.15$ & $30.67$ & $30.62$ & $29.61$ & $29.23$ & \checkmark \\
\multicolumn{2}{c}{DnCNN \cite{zhang2017beyond}} & $29.82$ & $29.24$ & $29.26$ & $28.88$ & $30.84$ & $30.95$ & $29.54$ & $29.17$ & \ding{55} \\
\multicolumn{2}{c}{BM3D \cite{makinen2020collaborativeBM3D}} & $30.05$ & $29.53$ & $29.93$ & $29.10$ & $31.08$ & $30.78$ & $29.56$ & $29.41$ & \ding{55}\\
\multicolumn{2}{c}{Proposed} & $\mathbf{30.86}$ & $\mathbf{30.33}$ & $\mathbf{30.31}$ & $\mathbf{30.14}$ & $\mathbf{31.72}$ & $\mathbf{31.69}$ & $\mathbf{30.42}$ & $\mathbf{30.09}$ & \checkmark \\
\hline
\end{tabular}   
\caption{\small Average PSNR values obtained by different denoisers plugged in the PnP-FB algorithm \eqref{e:algFBpap}, to solve the deblurring problem~\eqref{eq:invpb} with kernels of \cref{fig:kernels}(a)-(h) considering the BSD10 test set. The last row provides the average SSIM values for the observed blurred image $y$ in each experimental setting. Each algorithm is stopped after a fixed number of iterations equal to 1000. The best PSNR values are indicated in bold.}
\label{tab:FB_grayscale}
\end{table}

We first check the convergence of the PnP-FB algorithm considering the above-mentioned different denoisers. We study the quantity $(c_n)_{n\in \NN}$ defined in~\eqref{eq:crit_cn}, considering the inverse problem~\eqref{eq:invpb} with kernel in \cref{fig:kernels}(a). 
\cref{fig:cv_comp} shows the $c_n$ values with respect to the iterations $n \in \{1, \ldots, 1000\}$ of the PnP-FB algorithm for various denoisers $\net$: BM3D (\cref{fig:cv_comp}(a)), RealSN (\cref{fig:cv_comp}(b)), and the proposed firmly nonexpansive DnCNN (\cref{fig:cv_comp}(c)). 
On the one hand, as expected, PnP-FB with our network, which has been trained to be firmly nonexpansive, shows a convergent behavior with monotonic decrease of $c_n$. On the other hand,  we notice that the PnP-FB algorithm with BM3D or RealSN does not converge since $(c_n)_{n\in \NN}$ does not tend to zero, which confirms that neither BM3D nor RealSN are firmly nonexpansive. Note that, in the case of RealSN \cite{ryu2019plug}, nonexpansive-based constraints are also imposed on the network. Nevertheless, in practice these constraints are more restrictive than our approach (see \cref{rk:nonexp}\ref{rk:nonexp:iii}), and less accurate (see \cite[Assumption A]{ryu2019plug}).

In \cref{tab:FB_grayscale} we provide a quantitative analysis of the restoration quality obtained on the BSD10 dataset with the different denoisers. Although DnCNN and BM3D do not benefit from any convergence guarantees, we report the SNR values obtained after 1000 iterations. For all the eight considered kernels, the best PSNR values are delivered by the proposed firmly nonexpansive DnCNN. A possible explanation for the improved performance of our method compared to DnCNN, BM3D and RealSN is the stability of our PnP algorithm induced by the firm nonexpansiveness of our denoiser $\net$.

In \cref{fig:vis_sota} we show visual results and associated PSNR and SSIM values obtained with the different methods on the deblurring problem \eqref{eq:invpb} with kernel from \cref{fig:kernels}(a). We notice that despite good PSNR and SSIM values, the proximal methods yield reconstructions with strong visual artifacts (wavelet artifacts in \cref{fig:vis_sota}(c) and cartoon effects in \cref{fig:vis_sota}(d)). PnP-FB with BM3D provides a smoother image with more appealing visual results, yet some grid-like artifacts appear in some places (see e.g. red boxed zoom in \cref{fig:vis_sota}(e)). RealSN introduces ripple and dotted artifacts, while DnCNN introduces geometrical artifacts, neither of those corresponding to features in the target image. For this image, we can observe that our method provides better visual results as well as higher PSNR and SSIM values than other methods.

\begin{figure}[t]
\centering
\includegraphics[width=0.9\textwidth]{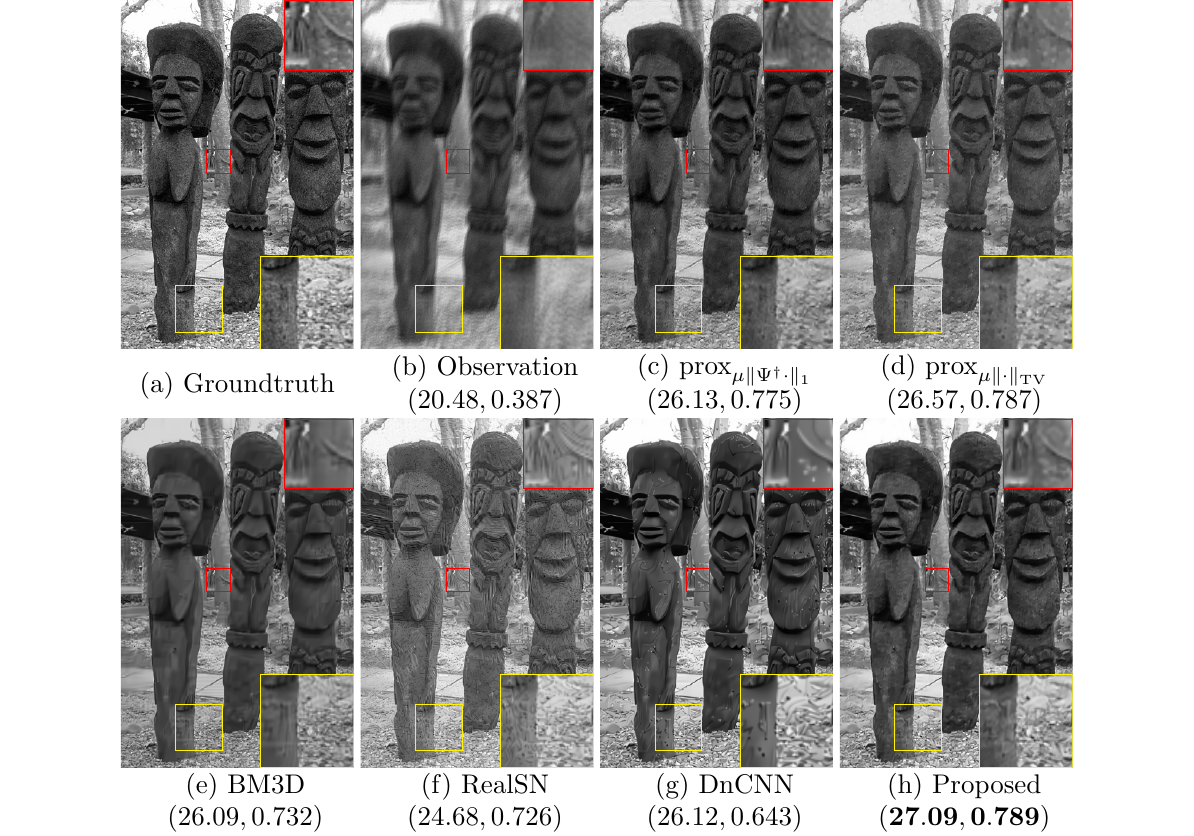}%

\vspace{-0.2cm}

\caption{\small
Reconstructions of an image from the BSD10 test set obtained with the PnP-FB algorithm~\eqref{e:algFBpap}, considering different denoisers as $\net$, for the deblurring problem with kernel from \cref{fig:kernels}(a) and $\nu=0.01$. Associated (PSNR, SSIM) values are indicated below each image, best values are highlighted in bold. Each algorithm is stopped after a fixed number of iteration equal to 1000.}
\label{fig:vis_sota}
\end{figure}

\begin{figure}
\centering
\includegraphics[width=0.8\textwidth]{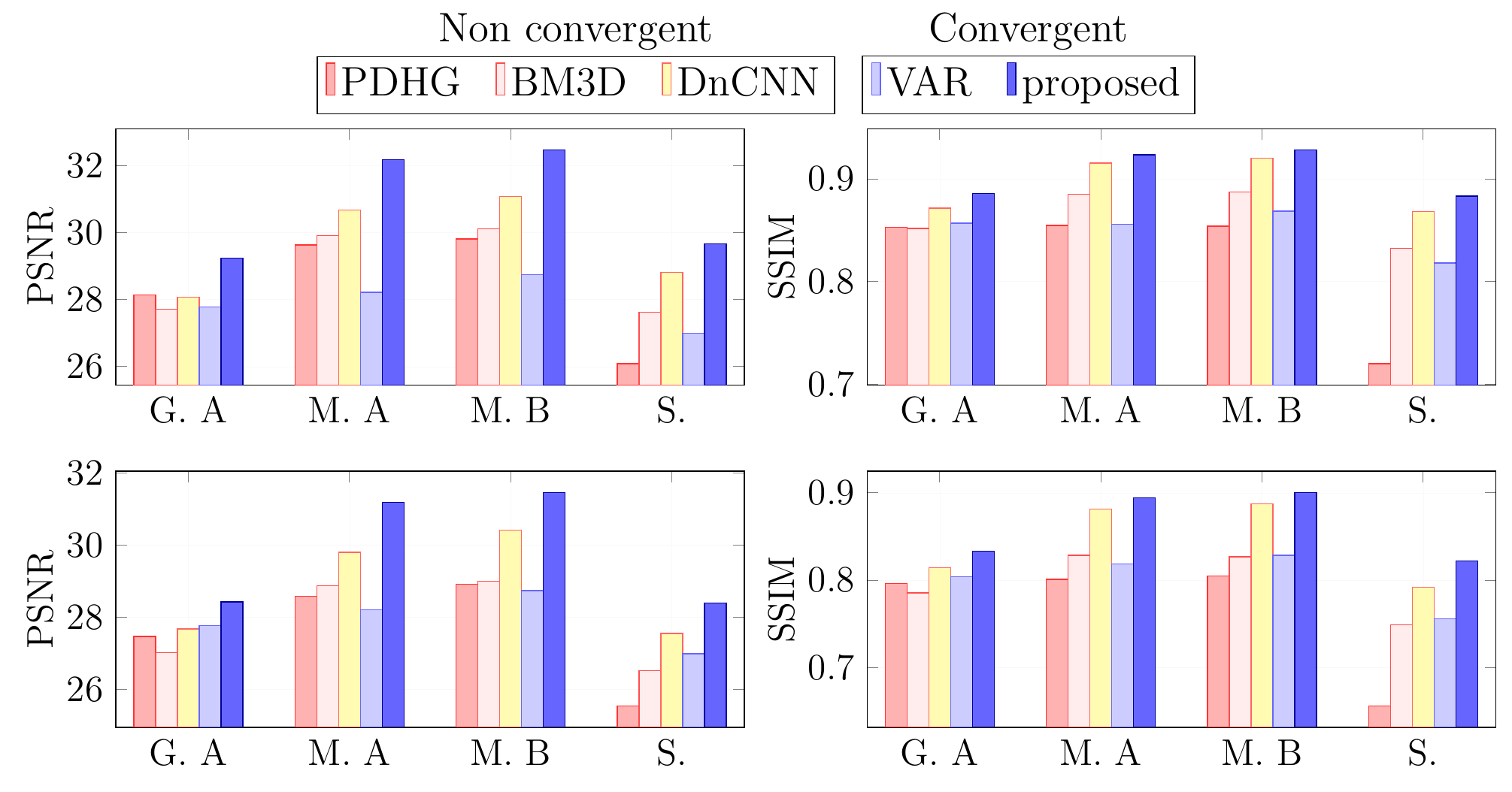}

\vspace{-0.2cm}

\caption{\small
Average PSNR and SSIM values obtained on the Flickr30 (top) and BSD500 (bottom) test sets using the experimental setups of \cite{bertocchi2020deep}: G. A, M. A, M. B, and S., for different methods.}
\label{fig:barplot}
\end{figure}

\begin{figure}
\centering
\includegraphics[width=0.90\textwidth]{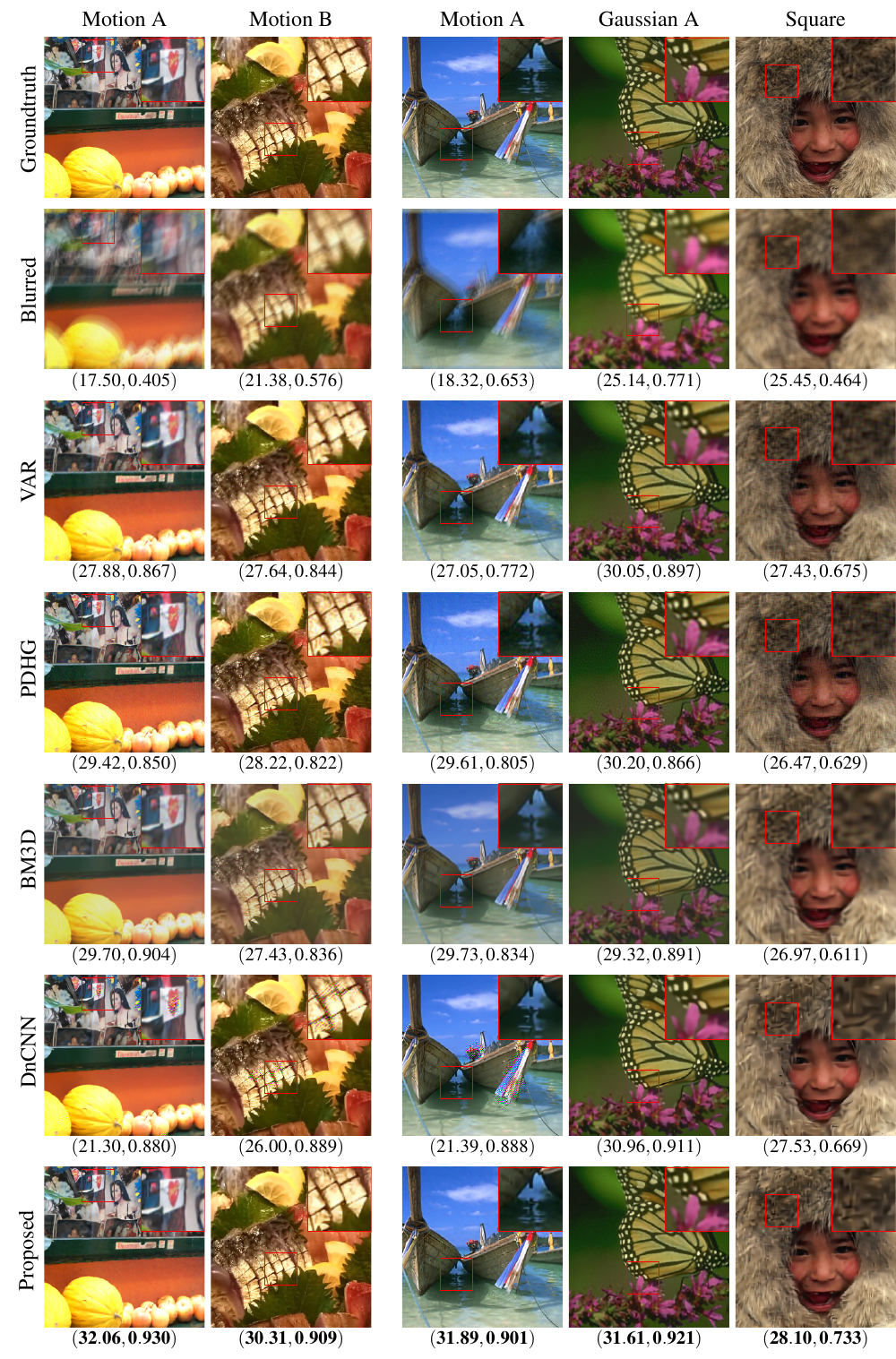}
\caption{\small 
Visual results on Flickr30 (first two columns) and BSD500 (last three columns) datasets for different methods (VAR, PDHG, PnP-FB with BM3D, PnP-FB with DnCNN and ours) for different deblurring problems with $\nu=0.01$. Associated (PSNR, SSIM) values are indicated below each image, best values are highlighted in bold. Except PDHG, each algorithm is stopped after 1000 iterations.}
\label{fig:vis_mc_1}
\end{figure}

The results presented in this section show that the introduction of the Jacobian regularizer in the training loss~\eqref{eq:loss_summarized} not only allows to build convergent PnP-FB methods, but also improves the reconstruction quality over both FB algorithms involving standard proximity operators, and existing PnP-FB approaches.

\paragraph{Color images} 

We now apply our strategy to a color image deblurring problem of the form \eqref{eq:invpb}, where the noise level and blurring operator are chosen to reproduce the experimental settings of \cite{bertocchi2020deep}, focusing on the four following experiments: First, the Motion A (M. A) setup with blur kernel (h) from \cref{fig:kernels} and $\nu=0.01$; second, the Motion B (M. B) setup with blur kernel (c) from \cref{fig:kernels} and $\nu=0.01$; third, the Gaussian A (G. A) setup with kernel (i) from \cref{fig:kernels} and $\nu=0.008$; finally, the Square (S.) setup with kernel (j) from \cref{fig:kernels} $\nu=0.01$. 
The experiments in this section are run on the Flickr30 dataset and on the test set from BSD500\footnote{As in \cite{bertocchi2020deep}, we consider $256\times 256$ centered-crop versions of the datasets.}. We compare our method on these problems with the variational method VAR from \cite{bertocchi2020deep}, and three PnP algorithms, namely PDHG \cite{meinhardt2017learning}, and the PnP-FB algorithm combined with the BM3D or DnCNN denoisers.  It is worth mentioning that, among the above mentioned methods, only the proposed approach and VAR have convergence guaranties. The results for PDHG and VAR are borrowed from \cite{bertocchi2020deep}.

For the proposed method, we choose $\gamma=1.99$ in the PnP-FB algorithm \eqref{e:algFBpap}, and we keep the same DnCNN architecture for $\net$ given in \cref{fig:arch_dncnn}, only changing the number of input/output channels to $C=3$. We first pretrain our network as described in \cref{Ssect:exp_settings}. We then keep on training it considering the loss given in~\eqref{eq:loss_summarized}, in which we set  $\varepsilon=5\times 10^{-2}$, $\lambda=10^{-5}$, and $\sigma=0.007$.

The average PSNR and SSIM values obtained with the different considered reconstruction methods, and for the different experimental settings, are reported in \cref{fig:barplot}. This figure shows that our method significantly improves reconstruction quality over the other considered PnP methods.

Visual comparisons are provided in~\cref{fig:vis_mc_1} for the different approaches. These results show that our method also yields better visual results. The reconstructed images contain finer details and do not show the oversmoothed appearance of PnP-FB with DnCNN or slightly blurred aspect of PnP-FB with BM3D. Note that, in particular, thanks to its convergence, the proposed method shows a homogeneous performance over all images, unlike PnP-FB with DnCNN that may show some divergence effects (see the boat picture for Motion A, row (f)). One can observe that the improvement obtained with our approach are more noticeable on settings M. A and M. B than on G. A and S. 

\section{Conclusion} 
\label{sec:5}

In this paper, we investigated the interplay between PnP \mbox{algorithms} and monotone operator theory, in order to propose a sound mathematical framework yielding both convergence guarantees and a good reconstruction quality in the context of computational imaging. 

First, we established a universal approximation theorem for a wide range of MMOs, in particular the new class of stationary MMOs we have introduced. This theorem constitutes the theoretical backbone of our work by proving that the resolvents of these MMOs can be approximated by building nonexpansive NNs. Leveraging this result, we proposed to learn MMOs in a supervised manner for PnP algorithms. A main advantage of this approach is that it allows us to characterize their limit as a solution to a variational inclusion problem. 

Second, we proposed a novel training loss to learn the resolvent of an MMO for high dimensional data, by imposing mild conditions on the underlying NN architecture. This loss uses information of the Jacobian of the NN, and can be optimized efficiently using existing training strategies. Finally, we demonstrated that the resulting PnP algorithms grounded on the FB scheme have good convergence properties when trained by leveraging the proposed Jacobian regularization. We showcased our method on an image deblurring problem and showed that the proposed PnP-FB algorithm outperforms both standard variational methods and state-of-the-art PnP algorithms. Interestingly, our results suggest that the optimal values for both the stepsize and training noise level can be determined automatically as functions of parameters of the data fidelity term, namely the Lipschitz constant $\mu$ of its gradient and the norm of the blurring kernel.

Note that the ability of approximating resolvents as we did is directly applicable to a much wider class of iterative algorithms than the forward-backward splitting \cite{combettes2020fixed}. In addition, we could consider a wider scope of applications than the restoration problems addressed in this work, such as image super-resolution or reconstruction, and future works also include the investigation of automatic methods for choosing $\lambda$.

\bibliographystyle{siamplain}
\bibliography{references}

\appendix

\section{Nonexpansive networks: additional results}
\label{apx:nonexpansive}

Based on the results in \cite{combettes2019}, useful sufficient conditions for a NN to be nonexpansive are given below:
\begin{proposition}\label{p:propnonexpan}
Let $Q$ be a feedforward NN as defined in \cref{mod:feedfwd}. Assume that, for every $m\in \{1,\ldots,M\}$, $R_{m}$ is $\alpha_{m}$-averaged with $\alpha_{m}\in [0,1]$. Then $Q$ is nonexpansive if one of the following conditions holds:
\begin{enumerate}
\item\label{p:propnonexpani} $\|W_{1}\|\cdots \|W_{M}\| \le 1$;
\item for every $m\in \{1,\ldots,M-1\}$, $\HH_{m}=\RR^{K_{m}}$ with $K_{m}\in \mathbb{N}\setminus\{0\}$, $R_{m}$ is a separable activation operator, in the sense that there exist real-valued one-variable functions $(\rho_{m,k})_{1 \le k \le K_{m}}$ such that, for every $x = (\xi_{k})_{1\le k \le K_{m}} \in \HH_{m}$,  $R_m(x)=(\rho_{m,k}(\xi_{k}))_{1\le k \le K_{m}}$, and
\begin{multline}\label{e:Lipboundsep}
(\forall \Lambda_{1}\in \mathcal{D}_{1,\{1-2\alpha_{1},1\}})\ldots (\forall \Lambda_{M-1}\in \mathcal{D}_{M-1,\{1-2\alpha_{M-1},1\}})\\
\|W_M
\Lambda_{M-1}\cdots \Lambda_1 W_1\| \le 1,
\end{multline}
where, for every $m\in \{1,\ldots,M-1\}$, $\mathcal{D}_{m,\{1-2\alpha_{m},1\}}$ denotes the set of diagonal $K_{m}\times K_{m}$ matrices with diagonal terms equal either to $1-2\alpha_{m}$ or $1$;
\item\label{p:propnonexpaniv} for every $m\in \{1,\ldots,M\}$, $\HH_{m}=\RR^{K_{m}}$ with $K_{m}\in \mathbb{N}\setminus\{0\}$ and $W_{m}$ is a matrix with nonnegative elements, $(R_{m})_{1\le m \le M-1}$ are separable activation operators, and
\begin{equation}\label{e:propnonexpaniv}
\|W_M\cdots W_1\| \le 1.
\end{equation}
\end{enumerate}
\end{proposition}
Note that the $\alpha$-averageness assumption on $(R_{m})_{1\le m \le M-1}$ means that, for every $m\in \{1,\ldots,M-1\}$, there exists a nonexpansive operator $\widetilde{R}_{m}\colon \HH_{m}\to \HH_{m}$ such that \linebreak${R_{m}=(1-\alpha_{m})\Id + \alpha_{m} \widetilde{R}_{m}}$.
Actually, most of the activation operators employed in neural networks (ReLU, leaky ReLU, sigmoid, softmax,...) satisfy this assumption with $\alpha_{m}= 1/2$ \cite{combettes2018deep}. 
A few others like the sorting operator used in max-pooling correspond to a value of the constant $\alpha_{m}$ larger than $1/2$ \cite{combettes2019}. It is also worth mentioning that, although Condition~\ref{p:propnonexpani} in \cref{p:propnonexpan} is obviously the simplest one, it is usually quite restrictive. If equation \eqref{e:propnonexpaniv} is weaker than Condition~\ref{p:propnonexpani}, Condition~\ref{p:propnonexpaniv} requires yet the network weights to be nonnegative as shown in \cite{combettes2019}.

By summarizing the results of the previous section, \cref{fig:NNmodel} shows a feedforward NN architecture for MMOs, for which \cref{p:propnonexpan} can be applied. It can be noticed that \eqref{e:JAQ} induces the presence of a skip connection in the global structure.

\begin{figure}
\centering
\includegraphics[width=0.95\textwidth]{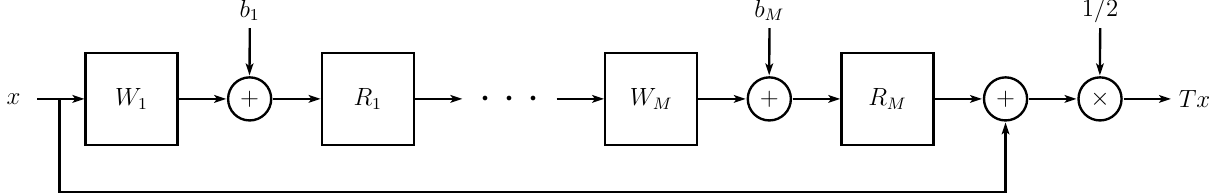}
\caption{\small Neural network modelling the resolvent of a maximally monotone operator. The weight operators $(W_{m})_{1\le m \le M}$ have to be set according to the conditions provided in \cref{p:propnonexpan}.}%
\label{fig:NNmodel}
\end{figure}

\section{Stationnary MMOs: additional proofs}
\label{apx:statMMO}

\begin{proof}[Proof of Example \ref{ex:stat:separabilityter}]
As $B$ is an MMO and $U$ is surjective, $U^*B U$ is an MMO  \cite[Corollary 25.6]{bauschke2017convex}. We are thus guaranteed
that $\ran(\Id+A)=\HH$ \cite[Theorem 21.1]{bauschke2017convex}.\\
For every $k\in \{1,\ldots,K\}$, let
\begin{align}
& D_{k}\colon \HH\to \HH_{k}\colon (x^{(\ell)})_{1\le \ell \le K}\mapsto x^{(k)}\label{e:defdec}\\
&\Pi_{k} = D_{k}U.\label{e:defdecap}
\end{align}
It can be noticed that
\begin{equation}
\sum_{k=1}^{K} \Pi_k^* \Pi_k = U^* U = \Id.
\end{equation}
Let $(p,q)\in\HH^2$. Every $(p',q')\in A(p)\times A(q)$ is such
\begin{align}
    p' = U^*r \text{ and } q' = U^* s, \label{e:ppUr-qpUs}
\end{align}
where $r\in B(Up)$ and $s\in B(Uq)$. Using \eqref{e:defdecap} and \eqref{e:ppUr-qpUs} yield, for every $k\in \{1,\ldots,K\}$,
\begin{equation}\label{e:maximalBkm}
    \scal{\Pi_k (p-q)}{\Pi_k(p'-q')}
    =  \scal{D_k Up-D_k Uq}{D_kr-D_ks}.
\end{equation}
Because of the separable form of $B$, $D_k r \in B_k(D_k Up)$ and $D_k s \in B_k(D_k Uq)$. It then follows from
\eqref{e:maximalBkm}  and the monotonicity of $B_k$ that 
\begin{equation}
    \scal{\Pi_k (p-q)}{\Pi_k(p'-q')} \ge 0.
\end{equation}
By invoking \cref{prop:immSMMO}\ref{prop:immSMMOii}, we conclude that $A$ is a stationary MMO.
\end{proof}

\begin{proof}[Proof of Example \ref{ex:stat:separabilityf}]
This corresponds to the special case of Example \ref{ex:stat:separabilityter} when, for every $k\in \{1,\ldots,K\}$, $\HH_k = \RR$ (see \cite[Theorem 16.47,Corollary 22.23]{bauschke2017convex}).
\end{proof}

\begin{proof}[Proof of Example \ref{ex:stat:linearity}]
If $B+B^*$ is nonnegative, $B$, hence $A$, are maximally monotone and  $J_{A}= J_B(\cdot-c)$ is firmly nonexpansive. As a consequence, the reflected resolvent of $B$, given by $Q = 2(\Id+B)^{-1}-\Id$, is nonexpansive. For every $k\in \{1,\ldots,K\}$, let $D_{k}$ be the decimation operator defined in \eqref{e:defdec} and let
\begin{align}
&\Pi_{k} = D_{k}\\
& \Omega_{k} = Q^* D_{k}^*D_{k}Q
\end{align}
$\Pi_{k}$ satisfies \eqref{e:stat21} and, since 
\begin{equation}
\Big\|\sum_{k=1}^{K} \Omega_{k}\Big\| = \|Q^*Q\| = \|Q\|^{2}\le 1,
\end{equation}
\eqref{e:stat22} is also satisfied.
In addition, for every $(x,y)\in \HH^{2}$ and, for every $k\in \{1,\ldots,K\}$, we have
\begin{align}
&\|\Pi_{k}\big(2J_{A} (x)-x -2J_{A} (y)+y\big)\|^2\nonumber\\ 
&= 
\|\Pi_{k}\big(2J_{B} (x-c)-x+c -2J_{B} (y-c)+y-c\big)\|^2\nonumber\\
&=  \scal{x-y}{\Omega_{k}(x-y)},
\end{align}
which shows that $A$ is a stationary MMO.\\
Note finally that, if $B$ is skewed or cocoercive linear operator, then $B+B^*$ is nonnegative.
\end{proof}

\begin{proof}[Proof of Example \ref{ex:stat:inverse}]
The resolvent of $A^{-1}$ is given by $J_{A^{-1}} = \Id - J_A$. In addition, since $A$ is stationary, there exist bounded linear operators $(\Pi_k)_{1\le k \le K}$ and self-adjoint operators $(\Omega_k)_{1\le k \le K}$ satisfying \eqref{e:stat1}-\eqref{e:stat22}.
For every $k\in \{1,\ldots,K\}$, we have then, for every $(x,y)\in \HH^2$,
\begin{align}
\|\Pi_k \big( 2 J_{A^{-1}}(x) - x - 2 J_{A^{-1}}(y) + y \big) \|^2
&=  \|\Pi_k \big( 2 J_{A}(y) - y - 2 J_{A}(x) + x \big) \|^2    \nonumber   \\
&\le \scal{y-x}{\Omega_k(y-x)}.
\end{align}
\end{proof}

\begin{proof}[Proof of Example \ref{ex:stat:lin}]
$B=\rho A(\cdot/\rho)$ is maximally monotone and its resolvent reads $J_B=\rho J_A(\cdot/\rho)$ \cite[Corollary 23.26]{bauschke2017convex}. Using the same notation as previously, for every $k\in \{1,\ldots,K\}$ and for every $(x,y)\in \HH^2$,
\begin{align}
\|\Pi_k \big( 2 J_B(x) - x - 2 J_B(y) + y \big) \|^2 
&=  \rho^2\left\|\Pi_k \Big( 2 J_{A}\Big(\frac{x}{\rho}\Big) - \frac{x}{\rho} - 2 J_{A}\Big(\frac{y}{\rho}\Big) + \frac{y}{\rho}\Big) \right\|^2    \nonumber   \\
&\le \scal{y-x}{\Omega_k(y-x)}.
\end{align}
\end{proof}

\section{Power iterative method}
\label{apx:powit}
The power iterative algorithm allows the spectral norm of a matrix $B$ to be computed efficiently \cite{golub2013matrix}. In our case, we apply this algorithm with $B=\jac \netav(x)$ at some point $x$. Notice that we do not need to compute or store the full Jacobian: in practice, at each iteration of this algorithm, we just need to apply $\jac \netav(x)$ to some vector and apply $\jac \netav(x)^\top$ to another vector. These products can be computed by automatic differentiation. \\
Due to memory limitations, all our experiments are performed with 5 iterations of the power method. We nevertheless observe in practice that this is sufficient to ensure convergence of the PnP-FB algorithm.

\end{document}